\documentclass[12pt]{article}
\usepackage[utf8]{inputenc}
\usepackage{geometry}
\usepackage[colorlinks]{hyperref}

\usepackage{graphicx}
\usepackage{tabularx}
\usepackage[normalem]{ulem}

\usepackage[bb=boondox]{mathalfa}  

\usepackage{amssymb,amsmath,amsthm}
\usepackage{mathtools}
\usepackage{mathrsfs}  
\usepackage{fancyhdr}
\usepackage{enumitem}
\usepackage{bbm}
\usepackage{dsfont}
\usepackage{tikz-cd}
\usepackage{pgfplots}
\pgfplotsset{compat=1.18}
\usepackage{yhmath}
\usepackage{cancel} 
\usepackage{ragged2e}
\usepackage{float}
\usepackage{extarrows}

\newcommand{\N}{\mathbb{N}}
\newcommand{\Z}{\mathbb{Z}}
\newcommand{\R}{\mathbb{R}}

\newcommand{\E}{\mathbb{E}}

\newcommand{\Prob}{\mathbb{P}}

\newtheorem{theorem}{Theorem}[section] 
\newtheorem{corollary}[theorem]{Corollary}
\newtheorem{proposition}[theorem]{Proposition}
\newtheorem{lemma}[theorem]{Lemma}

\theoremstyle{definition}
\newtheorem{definition}[theorem]{Definition}

\theoremstyle{remark}
\newtheorem{remark}[theorem]{Remark}

\title{Percolation on the stationary distributions of the voter model with stirring}
 
\author{Jhon Astoquillca \thanks{Bernoulli Institute, University of Groningen} \thanks{Groningen Cognitive System and Materials Center, University of Groningen} \thanks{Institute of Mathematics, Statistics and Computer Science, University of S\~ao Paulo, Brazil} \and Franco Severo \thanks{CNRS and Sorbonne Université, Paris, France} \and Réka Szabó${}^\ast$ \and Daniel Valesin \thanks{Department of Statistics, University of Warwick}} 
\date{\today}

\begin{document}

\maketitle

\begin{abstract}
The voter model with stirring is a variant of the classical voter model on $\Z^d$ with two possible opinions (0 and 1) that, in addition to copying neighbouring opinions at rate 1, allows voters to interchange their opinions at rate~$\mathsf{v}$ where~$\mathsf v \ge 0$ is the stirring parameter. This model was considered in \cite{Astoquillca24}, where it was proved that for~$d \ge 3$ and for any~$\mathsf{v}$ the set of extremal stationary measures is given by a family~$\{ \mu_{\alpha,\mathsf{v}}: \alpha \in [0,1] \}$, where~$\alpha$ is the density of voters with opinion~1. Sampling a configuration~$\xi$ from~$\mu_{\alpha, \mathsf v}$, we study~$\xi$ as a site percolation model on~$\Z^d$, where the set of occupied sites is the set of voters with opinion 1 in~$\xi$. Letting~$\alpha_c(\mathsf v)$ be the supremum of all the values of~$\alpha$ for which percolation does not occur~$\mu_{\alpha, \mathsf v}$-a.s., we prove that $\alpha_c(\mathsf{v})$ converges to~$p_c$, the critical density for classical Bernoulli site percolation, as~$\mathsf{v}$ tends to infinity. As a consequence, for $\mathsf v$ large enough, the model exhibits a non-trivial phase transition in~$\alpha$.
\end{abstract}
{\small

\textbf{\textit{Keywords:}} Voter model $\cdot$ Stirring process $\cdot$ Stationary distributions $\cdot$ Bernoulli site percolation $\cdot$ Duality $\cdot$ Annihilating random walks \\
  
  \textbf{\textit{Mathematics Subject Classification (2010)}} 60K35 $\cdot$ 82C22 $\cdot$ 82B43 
}


\section{Introduction}
The \emph{voter model} is a class of interacting particle systems that describes opinion dynamics. It was introduced by Clifford and Sudbury in~\cite{CliffordSudbury73} and Holley and Liggett in~\cite{HolleyLiggett75} and is treated in the expository texts~\cite{Liggett2005,Swart2017}.
In the standard model, each vertex (or site) represents an individual (or voter) which can have one of two possible opinions, corresponding to the states~0 and~1. At rate~1, each vertex updates its opinion by copying the opinion from one of its neighbours chosen uniformly at random. 

Since the introduction of the voter model, several variants have been proposed to incorporate more realistic or complex behaviours. One such variant is the \emph{voter model with stirring}, which augments the classical voter model by allowing neighboring voters to exchange opinions at a rate~$\mathsf v \ge 0$, the \textit{stirring parameter}. This stirring dynamics introduces additional randomness to the model and facilitates interactions between voters.

We consider the model on the $d$-dimensional Euclidean lattice and denote by 0 the lattice origin. For any~$d \in \N$ and real number~$\mathsf v \ge 0$, the voter model with stirring on~$\Z^d$ is a Markov process, denoted by~$(\xi_t)_{t \ge 0}$, with configuration space~$\{0,1\}^{\Z^d}$ whose stochastic dynamics can be described informally as follows. Each site~$x$ updates its opinion at rate~1 by choosing a site~$y$ uniformly at random from its~$2d$ nearest neighbours and adopting the opinion of~$y$. In addition, each pair of neighbouring sites~$\{x,y\}$ updates their opinions simultaneously at rate~$\mathsf v/(2d)$ by interchanging them. In Section~\ref{section_vm_with_stirring} we give a formal definition of the model.

The set of stationary measures of the model is described in~\cite{Astoquillca24}. For any fixed~$\mathsf v$ and~$d = 1$ or~$2$, the only extremal stationary measures are the consensus measures, which assign full mass to configurations that are identically equal to 0 or 1. For~$d \geq 3$ the extremal stationary measures form a family~$\{ \mu_{\alpha, \mathsf{v}}: \alpha \in [0,1] \}$. Each of these measures is translation-invariant, ergodic with respect to translations in~$\Z^d$ and has polynomial decay of correlations. Moreover, the density of~1's in $\mu_{\alpha, \mathsf{v}}$ is~$\alpha$, that is
$$ \mu_{\alpha, \mathsf v}( \xi \in \{0,1\}^{\Z^d} : \xi(0) = 1 ) = \alpha. $$

The objective of this paper is to study site percolation on configurations sampled from~$\mu_{\alpha, \mathsf v}$.  This problem was previously addressed for the classical voter model in \cite{RathValesin2017}, where the authors proved non-trivial percolation phase transition on the family of extremal stationary distributions~$\{\mu_\alpha:\alpha \in [0,1]\}$ of the model on~$\Z^d$ for~$d \geq 5$. That is, denoting by~$\mathrm{Perc} \subset \{0,1\}^{\Z^d}$ the set of configurations~$\xi \in \{0,1\}^{\Z^d}$ such that the subgraph of~$\Z^d$ induced by the set of occupied sites~$\{x \in \Z^d: \xi(x)=1\}$ has an infinite connected component, there exists a \textit{critical density}~$ \alpha_c \in (0,1)$ such that
\begin{equation}\label{eq_phase_transition_classic}
    \mu_\alpha(\mathrm{Perc}) = 0 \quad \text{for any } \alpha < \alpha_c \quad \text{and} \quad \mu_\alpha(\mathrm{Perc}) = 1 \quad \text{for any } \alpha > \alpha_c.
\end{equation}

Further progress in this direction was made in \cite{RathValesin2017, RathValesin17OnT}, where the authors considered a spread-out version of the voter model on~$\Z^d$, introducing a range parameter~$R$ that extends the interaction radius of the voters. They proved that for any~$d \ge 3$ and~$R$ large enough a nontrivial phase transition exists with a critical density~$\alpha_c(R)$, and this critical parameter approaches~$p_c$, the critical threshold for classical Bernoulli site percolation, as~$R$ tends to infinity.

In this work, we extend these studies to the voter model with stirring on~$\Z^d$. For any fixed~$\mathsf v \ge 0$ we have by ergodicity that~$\mu_{\alpha,\mathsf v}(\mathrm{Perc})$ is either 0 or 1. We define
$$ \alpha_c(\mathsf v) := \sup\{ \alpha \in [0,1]: \mu_{\alpha, \mathsf v}(\mathrm{Perc}) = 0 \}. $$
Since the family~$\mu_{\alpha, \mathsf v}$ is stochastically increasing in~$\alpha$, we have the analogous property as in~\eqref{eq_phase_transition_classic}. Our main results are the following.
\begin{theorem}\label{thm_main_2}
    For any~$d \ge 3$, as~$\mathsf{v} \to \infty$, the critical density value~$\alpha_c(\mathsf{v})$ for percolation phase transition of the stationary measures of the voter model with stirring parameter~$\mathsf v$ converges to the critical density value~$p_c$ for classical Bernoulli site percolation, that is, 
    $$ \lim_{\mathsf v \to \infty} \alpha_c(\mathsf v) = p_c. $$
\end{theorem}
Our main tools for proving Theorem~\ref{thm_main_2} are a multiscale renormalization scheme, introduced in Section~\ref{ss_renormalization}, and the decorrelation inequalities of Lemma~\ref{lemma_decomposition_main} and Lemma~\ref{lemma_no_same_E}. These inequalities play the role of a sprinkling argument and enable a comparison between~$\mu_{\alpha,\mathsf v}$ and~$\pi_{\alpha'}$, the Bernoulli($\alpha'$) product measure, for some~$\alpha'$ close to~$\alpha$.

It is well-known that~$p_c \in (0,1)$ for~$d \ge 2$ (see for instance~\cite{Grimmett99}). Using this fact, the following result is a straightforward consequence of Theorem~\ref{thm_main_2}.
\begin{corollary}\label{thm_main_1}
    For any $d \ge 3$ there exits $\mathsf v_0 = \mathsf{v}_0(d) \in (0,\infty)$ such that if $\mathsf v \ge \mathsf v_0$ then the family of stationary distributions of the voter model with stirring~$\{\mu_{\alpha,\mathsf v}:\alpha \in [0,1]\}$ exhibits a non-trivial percolation phase transition.
\end{corollary}

\subsection{Related work and organization of the paper}
In recent years, percolation models with long-range interactions exhibiting polynomial decay have attracted considerable interest. These models pose significant challenges, as many classical tools developed for Bernoulli percolation cannot be directly applied. For further details, see the Introduction of~\cite{Duminiletal23} and the references therein. The family of stationary measures of the voter model is a notable example within these models. We have previously referenced~\cite{RathValesin2017,RathValesin17OnT}, which study the site percolation phenomena on the classical voter model and the spread-out voter model.

The voter model with stirring has been previously studied in the literature~\cite{DeMasiFerrariLebowitz96,DurrettNeuhauser94,MiltonLandim2023}. Similarly, stirred variants of other interacting particle systems, such as the Ising model~\cite{Neuhauser1990} and the contact process~\cite{Katori94,Konno95,LevitValesin17,MytnikShlomov21}, have also been investigated.

The rest of the paper is organized as follows. In Section~\ref{section_2} we introduce the voter model with stirring on~$\Z^d$ $(d \ge 3)$ and its dual process, a \emph{system of coalescing-stirring particles}. Additionally, we describe the extremal stationary measures of the voter model with stirring and present a renormalization scheme used to analyze the percolation properties of this family of measures.

Section~\ref{section_3} is a lengthy section devoted to the proof of the decorrelation inequality of Lemma~\ref{lemma_decomposition_main}. The proof combines nearest-neighbor annihilating random walks, the renormalization scheme, and a coupling between stirring processes and independent random walks, detailed in the Appendix, Section~\ref{Section_Appendix}. Since several of the ideas developed there will later be extended in a more general setting, we present the argument in full detail. In Section~\ref{section_4}, we apply this lemma to a local percolation event that is increasing to prove that~$\liminf_{\mathsf v \to \infty}\alpha_c(\mathsf v) \ge p_c$. 

For the reverse inequality, we use a coarse-graining argument based on a local percolation event that is neither increasing nor decreasing. However, this event can be ``decomposed'' into an increasing-decreasing event and a decreasing-increasing event; see Section~\ref{section_5} for precise definitions. To handle these events, we require an extension of the decorrelation inequality from Lemma~\ref{lemma_decomposition_main}. This extension is developed in Section~\ref{section_5}, where we prove Lemma~\ref{lemma_no_same_E}. Finally, in Section~\ref{section_6}, we apply this result to prove that~$\limsup_{\mathsf v \to \infty}\alpha_c(\mathsf v) \leq p_c$, thereby completing the proof of Theorem~\ref{thm_main_2}.

\section{Preliminaries}\label{section_2}

\subsection{The model}\label{section_vm_with_stirring}
Given a vertex $x \in \Z^d$, we denote by~$|x|_1$ its~$\ell^1$-norm. Two vertices~$x,y$ are neighbors if~$|x-y|_1 = 1$; we denote this by~$x \sim y$. We denote by~$\{x,y\}$ an undirected edge between $x$ and $y$ and by~$(x,y)$ a directed edge from~$x$ to~$y$. \\

\noindent \textbf{The voter model with stirring.} Given $d \in \N$ and $\mathsf{v} \in [0,\infty)$, the \textit{voter model with stirring} on~$\Z^d$, denoted by~$(\xi_t)_{t \geq 0}$, is a Feller process with state space $\{0,1\}^{\Z^d}$ and pre-generator
$$\mathcal{L} f(\xi) = \sum_{ \substack{\{x,y\}:x,y \in \Z^d,\\ x \sim y}} \frac{\mathsf{v}}{2d} \cdot \big[ f(\xi^{x,y}) - f(\xi)\big] + \sum_{ \substack{(x,y):x,y \in \Z^d,\\ x \sim y}}   \frac{1}{2d}\cdot \big[ f(\xi^{y \to x}) - f(\xi) \big],$$
where $f:\{0,1\}^{\Z^d} \to \R$ is a function that only depends on finitely many coordinates, and for any~$x,y \in \Z^d$,~$\xi \in \{0,1\}^{\Z^d}$ we write
$$\xi^{x,y}(z) = \left\{ \begin{array}{ll}
\xi(z) & \text{if } z \notin \{x,y\},  \\
    \xi(x) & \text{if } z = y, \\
    \xi(y) & \text{if } z = x \end{array} \right. \hspace{2mm} \text{ and }\hspace{2mm} \xi^{y \to x}(z) = \left\{ \begin{array}{ll}
    \xi(z) & \text{if } z \neq y,  \\
    \xi(x) & \text{if } z = y. \end{array} \right.$$

Given~$\xi \in \{0,1\}^{\Z^d}$, we denote by~$P_\xi$ a probability measure under which~$(\xi_t)_{t \ge 0}$ is defined and satisfies~$P_\xi(\xi_0 = \xi) = 1$. \\

\noindent \textbf{System of coalescing-stirring particles.} We now present a particle system that satisfies a duality relation with the voter model with stirring. To do so, we first introduce some notation. 

For~$x,y,z \in \Z^d$, we write
\begin{equation*}
	z^{x\to y}:= \begin{cases}
		z&\text{if } z \neq x,\\ y&\text{if } z = x,
	\end{cases} \qquad \text{and} \qquad 
	z^{x\leftrightarrow y}:= \begin{cases}
		z&\text{if } z \notin \{x,y\},\\ y&\text{if } z = x,\\x&\text{if } z = y.
	\end{cases}
\end{equation*}
Fix~$n \in \N$, for any~$\mathbf{y} = (y_1,\dots, y_n) \in (\Z^d)^n$ we write
\[
	\mathbf{y}^{x\to y}=(y_1^{x \to y},\ldots, y_n^{x\to y}) \qquad \text{and}\qquad\mathbf{y}^{x\leftrightarrow y}=(y_1^{x \leftrightarrow y},\ldots, y_n^{x \leftrightarrow y}).
\]

Let~$\mathsf v \ge 0$ and~$\mathbf x=(x_1,\ldots,x_n) \in (\Z^d)^n$. We define a \emph{system of coalescing-stirring particles}~$(\mathbf{Y}_t)_{t \ge 0}=(Y^{x_1}_t,\ldots,Y^{x_n}_t)_{t \ge 0}$ as the continuous-time Markov chain on~$(\Z^d)^n$ started at~$\mathbf x$ with generator
\begin{align}\label{eq_lvcs}
	\mathcal L^{(\mathsf v)}_{\mathrm{cs}}f(\mathbf y) &= \mathcal L_{\mathrm{coal}}f(\mathbf{y}) + \mathcal L_{\mathrm{stir}}^{(\mathsf{v})} f(\mathbf{y}),
\end{align}
where
\begin{equation}\label{generator_L_v_stirring}
    \mathcal L_{\mathrm{coal}} f(\mathbf{y}) = \hspace{-2mm} \sum_{\substack{(x,y): x,y \in \Z^d:\\ x\sim y}} \hspace{-3mm} \frac{f(\mathbf{y}^{x \to y}) - f(\mathbf{y})}{2d},\quad \mathcal L_{\mathrm{stir}}^{(\mathsf v)}f(\mathbf{y}) = \hspace{-2mm} \sum_{\substack{\{x,y\}: x,y \in \Z^d:\\ x\sim y}} \hspace{-3mm} \frac{\mathsf{v} \cdot (f(\mathbf{y}^{x \leftrightarrow y}) - f(\mathbf{y}))}{2d}.    
\end{equation}

We see from the definition above that each $(Y^{x_i}_t)_{t \geq 0}$ is a continuous-time simple random walk on~$\Z^d$ with rate $1+\mathsf{v}$, i.e., the holding times between jumps are i.i.d. Exp($1+\mathsf v$) random variables and when a jump occurs at time $t$ and site $Y^{x_i}_{t^-} = x \in \Z^d$, then~$Y^{x_i}_t$ is uniformly distributed on the nearest neighbours of~$x$. Moreover, for~$i \neq j$, when~$Y^{x_i}_t$ and~$Y^{x_j}_t$ are in neighboring positions, say~$Y^{x_i}_t=w\sim z=Y^{x_j}_t$, they can interact in two possible ways. First, with rate~$\mathsf v/(2d)$ they exchange their positions. Second, with rate~$1/(2d)$ the particle at~$w$ jumps on top of the particle at~$z$, and with rate~$1/(2d)$ the particle at~$z$ jumps on top of the particle at~$w$; in either case, they coalesce, and move together from then onward.

Given~$n \in \N$ and~$\mathbf{x} \in (\Z^d)^n$, we denote by~$P_{\mathbf x}$ a probability measure under which~$(\mathbf{Y}_t)_{t \ge 0}$ is defined and satisfies~$P_\mathbf{x}(\mathbf{Y}_0 = \mathbf{x}) = 1$.\\

\noindent \textbf{Duality.} The voter model with stirring and the system of coalescing-stirring particles satisfy the following duality relation, stated in Section~3.1 of~\cite{Astoquillca24}. For all $n \in \N$, $\xi \in \{0,1\}^{\Z^d}$ and $\mathbf{x}=(x_1,\dots,x_n) \in (\Z^d)^n$,
\begin{equation}\label{eq_duality}
	P_\xi( \xi_t \equiv 1 \text{ on } \{x_1,\ldots, x_n\}) = P_{\mathbf x}( \xi \equiv 1 \text{ on } \{Y^{x_1}_t,\ldots, Y^{x_n}_t\}). 
\end{equation}
Let~$|\mathbf Y_t|$ denote the cardinality of the set~$\{Y^{x_1}_t,\ldots, Y^{x_n}_t\}$ and let~$\pi_\alpha$ denote the product Bernoulli($\alpha$) measure on~$\{0,1\}^{\Z^d}$. Integrating both sides of~\eqref{eq_duality}, we have
\begin{equation}
	\label{eq_duality_bernoulli}
	\int \pi_\alpha(\mathrm{d}\xi)\; P_\xi( \xi_t \equiv 1 \text{ on } \{x_1,\ldots, x_n\}) = E_{\mathbf x}[\alpha^{|\mathbf Y_t|}],
\end{equation}
where~$E_{\mathbf x}$ is the expectation operator associated with~$P_{\mathbf x}$. Observe that~$t \mapsto |\mathbf Y_t|$ is non-increasing, so we can define~$|\mathbf Y_\infty|:=\lim_{t \to \infty} |\mathbf Y_t|$. This then shows that, when the voter model with stirring is started from~$\pi_\alpha$ and time is taken to infinity, its law converges in distribution to a measure~$\mu_{\alpha,\mathsf v}$ characterized by the equality
\[\mu_{\alpha,\mathsf v}(\xi:\xi \equiv 1 \text{ on } \{x_1,\ldots,x_n\}) = E_{\mathbf{x}}[\alpha^{|\mathbf Y_\infty|}].\]
Since this measure is obtained as a distributional limit of the voter model with stirring as~$t \to \infty$, it is stationary for the dynamics.

The following procedure gives a way to sample from the measure~$\mu_{\alpha,\mathsf v}$ inside a finite set~$\Lambda = \{x_1,\ldots,x_n\} \subset \Z^d$. Letting~$\mathbf x=(x_1,\ldots,x_n)$ we use the process~$(\mathbf Y_t)_{t \ge 0}$ started from~$\mathbf x$ to obtain a random partition~$\mathscr C$ of~$\Lambda$ by declaring that~$x_i$ and~$x_j$ are in the same partition block if there exists~$t\geq 0$ such that~$Y^{x_i}_t=Y^{x_j}_t$ (that is, the two particles eventually coalesce). Conditionally on a realization of this partition, say~$\mathscr C=\{\mathcal C_1,\ldots, \mathcal C_m\}$, we sample independent Bernoulli($\alpha$) random variables,~$\mathcal B_1,\ldots,\mathcal B_m$, and then construct~$\{\xi(x):x \in \Lambda\}$ as follows. For each~$i=1,\dots,m$ we set~$\xi|_{\mathcal C_i} \equiv \mathcal B_i$, where~$\xi|_{\mathcal C_i}$ denotes the restriction of~$\xi$ to~$\mathcal C_i$.

It is easy to check that this procedure gives a measure which can be extended by a consistency property to a measure on~$\{0,1\}^{\Z^d}$. This measure in turn coincides with~$\mu_{\alpha,\mathsf v}$, because they coincide on sets of the form~$\{\xi: \xi \equiv 1 \text{ on } \Lambda\}$ for any finite~$\Lambda$.

We now present some properties of the family~$\mu_{\alpha,\mathsf v}$, $\alpha \in [0,1]$. Recall that a stationary measure of an interacting particle system is said to be \emph{extremal} if it cannot be expressed as a non-trivial convex combination of other stationary measures. 

\begin{lemma}\label{familia_mu_alpha_properties}
For any~$\mathsf{v} \ge 0$ the family $\{ \mu_{\alpha,\mathsf v}: \alpha \in [0,1] \}$ has the following properties:
\begin{itemize}
    \item[$\mathbf{(A)}$] for $d=1,2$, $\mu_{0,\mathsf v}$ and $\mu_{1,\mathsf v}$ are the only extremal stationary measures of the voter model with stirring, and for~$d \ge 3$ the family~$\{\mu_{\alpha,\mathsf v}: \alpha \in [0,1]\}$ coincides with the set of extremal stationary measures;
    \item[$\mathbf{(B)}$] for each $\alpha \in [0,1]$ the measure $\mu_{\alpha,\mathsf v}$ is invariant and ergodic under translations of~$\Z^d$;
    \item[$\mathbf{(C)}$] if $\alpha \leq \alpha'$, then $\mu_{\alpha,\mathsf v}$ is stochastically dominated by~$\mu_{\alpha', \mathsf v}$;
    \item[$\mathbf{(D)}$] the law of $1-\xi$ under $\mu_{\alpha,\mathsf v}$ is the same as the law of $\xi$ under $\mu_{1-\alpha,\mathsf v}$. 
\end{itemize}
\end{lemma}
\begin{proof}
\textbf{(A)} and \textbf{(B)} are proved in Theorem~1.1 and Proposition~1.2 of~\cite{Astoquillca24}, respectively. Properties \textbf{(C)} and \textbf{(D)} are readily obtained from the above construction involving the random partition~$\mathscr C$ of any finite set~$\Lambda \subset \Z^d$.

\end{proof}

\subsection{Renormalization scheme}\label{ss_renormalization}
Given a vertex~$x \in \Z^d$, we denote by~$|x|$ its~$\ell^\infty$-norm. We say that two vertices~$x,y$ are~$\ast$\emph{-neighbours} if~$|x-y|=1$. A $\ast$-\textit{connected path} is a sequence~$\gamma = (\gamma(0), \gamma(1),\dots)$ such that $\gamma(i)$ and $\gamma(i+1)$ are $\ast$-neigbours for each~$i \ge 0$. Likewise, a \emph{connected path} is a sequence~$\gamma = (\gamma(0),\gamma(1),\dots)$ such that~$\gamma(i) \sim \gamma(i+1)$ for each~$i \ge 0$.

Given disjoint sets $A, B \subseteq \Z^d$ and a configuration $\xi \in \{0,1\}^{\Z^d}$, we write $A
\xleftrightarrow[]{\ast \xi} B$, if there exists a~$\ast$-connected path from a $\ast$-neighbor of a point in $A$ to a $\ast$-neighbor of a point in $B$ and $\xi$ is equal to 1 on all points in this path. Analogously, we write $A
\xleftrightarrow[]{\xi} B$, if there exists a connected path from a neighbor of a point in~$A$ to a neighbor of a point in~$B$ and~$\xi$ is equal to 1 on all points in this path. Note that  $A\xleftrightarrow[]{\xi} B$ implies  $A\xleftrightarrow[]{\ast\xi} B$.

The ball of radius $R$ corresponding to the~$\ell^\infty$-norm is given by
\begin{align*}
    & B(R) := \{x \in \Z^d: |x| \leq R\}, \quad B(x,R) := \{y \in \Z^d: |x-y| \leq R \}, \qquad R\in\R^+, x \in \Z^d. 
\end{align*}


Given~$A,B \subseteq \Z^d$, we define the distance between the sets~$A$ and~$B$ as
$$ \mathrm{dist}(A,B): = \min\{ |x-y|: x \in A, y \in B \}. $$
We will consider spatial boxes of scales
\begin{equation}\label{scales_mathcal_L0}
    L_N := L \cdot \ell^N, \quad L,\ell \in \N,\; N\in\N\cup \{0\},
\end{equation}
(that is, $L_0=L$) and use them to study the percolation event defined as
\begin{equation}\label{Perc_def_set}
    \mathrm{Perc} := \left\{ \begin{array}{c}
    \xi \in \{0,1\}^{\Z^d}: \text{ there exists an infinite self-avoiding connected path} \\
    \gamma = (\gamma(0), \gamma(1),\dots) \text{ such that } \xi(\gamma(i)) = 1 \text{ for each } i\geq 0
\end{array}   \right\}.
\end{equation}

The following statement is proved in Section~4.1 of~\cite{RathValesin2017}. As an abuse made to simplify notation, for~$L=1$, we interpret the event~$B(x,L-2) \xleftrightarrow[]{\; \ast \xi \;} B(x,2L)^c$ as meaning simply that~$\xi(x) = 1$. \\

Let~$\mu$ be a measure on~$\{0,1\}^{\Z^d}$ that is invariant under translations of~$\Z^d$. Then,
\begin{equation}\label{eq_part_1_imply_noperc}
    \begin{aligned}
    &\text{if there is a choice of } L \text{ and } \ell \text{ in } \eqref{scales_mathcal_L0} \text{ such that} 
    \\[.2cm]
    &\mu\big( B(L_N - 2) \xleftrightarrow{\; \ast \xi \;} B(2L_N)^c \big) \leq 2^{-2^N} \text{ for any } N \in \N, \text{ then } \mu(\mathrm{Perc})=0;
    \end{aligned}
\end{equation}
and
\begin{equation}\label{eq_part_2_imply_noperc}
    \begin{aligned}
    &\text{if there is a choice of } L \text{ and } \ell \text{ in } \eqref{scales_mathcal_L0} \text{ such that} 
    \\[.2cm]
    & \mu\big( B(L_N - 2) \xleftrightarrow{\; \ast (1-\xi) \;} B(2L_N)^c \big) \leq 2^{-2^N} \text{ for any } N \in \N, \text{ then } \mu(\mathrm{Perc})=1.
    \end{aligned}
\end{equation}

The main tool to prove these facts is a multi-scale renormalization scheme introduced in~\cite{Rath15} and also used in~\cite{RathValesin2017}. We now explain this scheme and recall some of its properties. Fix $L\in \N$ and~$\ell \ge 6$, define~$(L_N)_{N \ge 0}$ as in~\eqref{scales_mathcal_L0}, and define the sublattice
\begin{equation}\label{scales_mathcal_L}
    \mathcal L_N := L_N \cdot \Z^d, \quad N \in \N \cup \{0\}.
\end{equation}
For any $k \in \N \cup \{0\}$ let $T_{(k)} := \{1,2\}^k$ (with~$T_{(0)} := \{\varnothing\}$) and let 
$$ T_N := \cup^N_{k = 0} T_{(k)} $$
be the \textit{binary tree of height} $N$. If $0 \leq k \leq N$ and $m = (\eta_1,\dots,\eta_k) \in T_{(k)}$, we let
$$ m_1 := (\eta_1,\dots,\eta_k,1) \quad\text{and}\quad m_2 := (\eta_1,\dots,\eta_k,2) $$
be the two children of $m$ in $T_{(k+1)}$.

\begin{definition}
    $\mathcal T: T_N \to \Z^d$ is a \textit{proper embedding} of $T_N$ if 
    \begin{enumerate}
        \item $\mathcal T(\{ \varnothing \}) = 0$;
        \item for all $0 \leq k \leq N$ and $m \in T_{(k)}$ we have $\mathcal T(m) \in \mathcal L_{N-k}$;
        \item for all $0 \leq k \leq N$ and $m \in T_{(k)}$ we have 
        $$|\mathcal T(m_1)- \mathcal T(m)| =  L_{N-k}, \qquad |\mathcal T(m_2)- \mathcal T(m)| = 2L_{N-k}.$$
    \end{enumerate}
    We denote by $\mathcal P_N$ the set of proper embeddings of $T_N$ into $\Z^d$.
\end{definition}
We now collect three facts concerning proper embeddings from~\cite{Rath15}. While these are originally proved for~$\ell = 6$, the arguments hold for any~$\ell \ge 6$.\\

\noindent \textbf{Bound on the number of proper embeddings.} There exists a dimension-dependent constant~$C_d>0$ such that
    \begin{equation}\label{bound_for_Lambda_N}
        |\mathcal P_N| \leq (C_d \cdot \ell^{2d-2})^{2^N} \quad \text{ for any } N \in \N.    
    \end{equation}
This fact follows from Lemma~3.2 in~\cite{Rath15}. \\
    
\noindent \textbf{Proper embeddings and crossing events.} For any~$N \in \N$, 
\begin{equation}\label{path_anulus}
    \{ B(L_N-2) \xleftrightarrow[]{\; \ast \xi \;} B(2L_N)^c \} \subset \bigcup_{\mathcal T \in \mathcal P_N} \bigcap_{ m \in T_{(N)}} \{ B(\mathcal T (m), L_0-2) \xleftrightarrow[]{\; \ast \xi \;} B(\mathcal T(m), 2L_0)^c \}.
\end{equation}
The inclusion also holds if the~$\ast$-connected path is replaced with a connected path on both sides of the inclusion, as shown in the proof of~\eqref{path_anulus} in Lemma~3.3 of~\cite{Rath15}. The inclusion asserts that if we have a crossing of the annulus~$B(2L_N) \setminus B(L_N)$ of scale~$L_N$, then there is a proper embedding~$\mathcal T \in \mathcal P_N$ such that all the annuli~$\{B(\mathcal T(m),2L_0) \setminus B(\mathcal T(m),L_0), m \in T_{(N)}\}$ of scale~$L_0$ are crossed. For an illustration of this statement see Figure~\ref{fig_leaves} below or Figure~2 in~\cite{RathValesin2017}. 


\begin{figure}[h]
    \centering
    \includegraphics[width=0.75\linewidth]{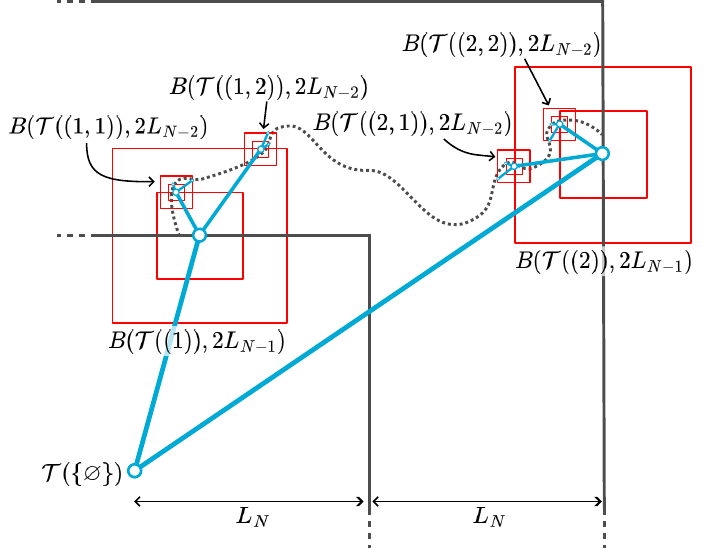}
    \caption{Illustration of a crossing path of scale~$L_N$ and the associated binary tree embedded in $\Z^d$ by~$\mathcal T$.}
    \label{fig_leaves}
\end{figure}

\noindent \textbf{Leaves are spread-out on all scales}. For any $\mathcal T \in \mathcal P_N$ and any $m_0 \in T_{(N)}$, we have
\begin{equation}\label{spread_out_all_scales}
    \big| \big\{ m \in T_{(N)}: \mathrm{dist}\big(B(\mathcal T(m_0)),2L_0),B(\mathcal T(m),2L_0)\big) \leq \ell^k L/2 \big\} \big| \leq 2^{k-1}, \quad k \geq 1.
\end{equation}
This follows from Lemma~3.4 in~\cite{Rath15}. The statement guarantees that in a proper embedding the set of the images of the leaves~$\{ \mathcal T(m): m \in T_{(N)} \}$ is spread-out on all scales.

\section{Decorrelation inequality}\label{section_3}
In this section we study percolation properties of the measure~$\mu_{\alpha,\mathsf v}$ by comparing it with the product Bernoulli($\alpha$) measure~$\pi_\alpha$. We state our main result in Lemma~\ref{lemma_decomposition_main} below. We first need to introduce some notation.

Given~$\Lambda \subset \Z^d$ finite, we let~$\mathcal G_\Lambda$ be the set of subsets of~$\{0,1\}^{\Z^d}$ given by
$$E = \{ \xi \in \{0,1\}^{\Z^d} : \xi |_\Lambda \in E_\Lambda \}, $$
where~$E_\Lambda \subseteq \{0,1\}^\Lambda$. The elements of
$\mathcal G_\Lambda$
are called \emph{cylinder sets}. Given~$E \in \mathcal G_\Lambda$, we say that~$E$ is an~\emph{increasing} event if~$[ \xi \in E,\; \xi(x) \leq \xi'(x) \; \forall x \in \Lambda ]$ implies~$\xi' \in E$, and a~\emph{decreasing} event if~$[ \xi \in E,\; \xi(x) \geq \xi'(x) \; \forall x \in \Lambda ]$ implies~$\xi' \in E$.

Given a vector~$x \in \Z^d$, a set $\Lambda \subseteq \Z^d$ and a configuration~$\xi \in \{0,1\}^{\Lambda}$, we write~$\theta_x \xi$ for the configuration in~$\{0,1\}^{x + \Lambda}$ defined by
$$ (\theta_x \xi)(y) := \xi(y-x), \quad y \in x+ \Lambda.$$
Given a cylinder set~$E$, we write~$\theta_x E := \{ \theta_x \xi: \xi \in E \}$. 

Finally, we define the following constants. Given~$N, \ell, L \in \N$ and~$\alpha \in (0,1)$ we set
\begin{equation}\label{c_ell_L_0_alpha_def}
        c(\ell,L,\alpha,N) := \exp \big\{5c_\mathrm{nn} \cdot  \alpha^{-2} \cdot |B(2L)|^2 \cdot 2^N \cdot \ell^{(2-d)/4} \big\} 
    \end{equation}
and
\begin{equation}\label{q_def}
q(\ell,L) := (3\sqrt{2})^d \cdot c_d \cdot |B(2L)| \cdot \ell^{-d/4},    
\end{equation}
where the constants~$c_\mathrm{nn}$ and~$c_d$ will be given later, in Remark~\ref{remark_value_h_c_1} and Lemma~\ref{lemma_stochastic_domination}, respectively; they only depend on the dimension, and have to do with standard random walk estimates.

\begin{lemma}\label{lemma_decomposition_main}
There exist dimension-dependent constants~$a,b>0$ such that the following holds. Fix~$N, \ell,L \in \N$ with~$\ell \ge 6$ and~$q(\ell,L) \in (0,1)$ and a proper embedding~$\mathcal T \in \mathcal P_N$. Then, there exists~$\mathsf v_0 = \mathsf v_0(\ell)$ such that, for all~$\mathsf{v} \ge \mathsf{v}_0$, ~$\alpha \in (0,1)$, and all increasing events~$E$ and decreasing events~$F$ in~$\mathcal{G}_{B(2L)}$, we have
    \begin{equation}\label{mu_int_theta_incr_decr}
    \begin{aligned}
    & \mu_{\alpha,\mathsf v} \Big( \bigcap_{m \in T_{(N)}} \mathcal \theta_{\mathcal T(m)}E \Big) \leq c(\ell, L,\widehat{\alpha},N) \cdot \sum_{A \subseteq T_{(N)} } \big( \Psi_\mathrm{incr}(E) \big)^{|T_{(N)} \setminus A|} \cdot ( \Psi_1 )^{|A|}, \\
    & \mu_{\alpha,\mathsf v} \Big( \bigcap_{m \in T_{(N)}} \mathcal \theta_{\mathcal T(m)} F \Big) \leq c(\ell, L,\widehat{\alpha},N) \cdot \sum_{A \subseteq T_{(N)} } \big(\Psi_\mathrm{decr}(F) \big)^{|T_{(N)} \setminus A|} \cdot ( \Psi_1 )^{|A|},
    \end{aligned}
\end{equation}
    where~$\widehat{\alpha}:= \min\{\alpha,1-\alpha\}$ and
    \begin{align}
        \label{Psi_1_def} & \Psi_1 = |B(2L)| \cdot \Big( 2d \Big( 1 + \frac{d}{4} \Big)^{-\ell/8} + \exp\{-\ell\} \Big), \\[0.2cm]
        \label{Psi_incr_def} & \Psi_\mathrm{incr}(E) = \pi_{\alpha + q-\alpha q}(E) + a \cdot |B(2L)|^2\exp \Big\{ -\frac{b \cdot \ell^{1/4}}{|B(2L)|} \Big\}, \\[0.2cm]
        \label{Psi_decr_def} & \Psi_\mathrm{decr}(F) =  \pi_{\alpha (1-q)}(F) + a \cdot |B(2L)|^2\exp \Big\{ -\frac{b \cdot \ell^{1/4}}{|B(2L)|} \Big\}.
    \end{align} 
\end{lemma}

In Section~\ref{section_4}, we use this lemma to prove the asymptotic lower bound in Theorem~\ref{thm_main_2}, with a suitable choice of the renormalization scales~$\ell$ and~$L$. The upper bound requires a generalization of this lemma, which is developed in Section~\ref{section_5}.

We prove Lemma~\ref{lemma_decomposition_main} over the next three subsections. First, in Subsection~\ref{ss_proof_proposition}, we provide an upper bound for increasing and decreasing events under~$\mu_{\alpha, \mathsf v}$ (Lemma~\ref{lemma_mu_stir_pi_alpha}) using an annihilation approach similar to that in Lemma~2.5 of~\cite{RathValesin17OnT}. Next, in Subsection~\ref{ss_translation_inc_dec}, we utilize the renormalization scheme and Lemma~\ref{lemma_mu_stir_pi_alpha} to derive~\eqref{ineq_mu_stir_int_C_m}, an upper bound under~$\mu_{\alpha,\mathsf v}$ on the joint occurrence of increasing and decreasing events in the boxes centered at the leaves~$\{\mathcal T(m), m \in T_{(N)}\}$. Finally, in Subsection~\ref{ss_decomposition}, we complete the proof of Lemma~\ref{lemma_decomposition_main} using a coupling between stirring random walks and independent random walks.

\subsection{Graphical construction and consequences}\label{ss_proof_proposition}
We first introduce a useful Poissonian \emph{graphical construction} to sample a system of coalescing-stirring particles, as well as a random configuration~$\xi^{(\alpha)}$ with distribution~$\mu_{\alpha,\mathsf v}$. For every pair of neighboring vertices~$x,y \in \Z^d$, let~$\mathcal{J}^{(x,y)}$, $\mathcal{J}^{(y,x)}$ and~$\mathcal{H}^{\{x,y\}}$ be Poisson point processes on~$[0,\infty)$ with rates~$\tfrac{1}{2d}$, $\tfrac{1}{2d}$ and~$\tfrac{\mathsf{v}}{2d}$, respectively. We denote by $\Prob$ a probability measure under which all these processes are defined and are independent. 

In this probability space, we define a \emph{system of coalescing-stirring particles}~$\{ Y^x_t: x \in \Z^d,\; t \ge 0\}$ as follows. For each~$x \in \Z^d$,~$(Y^x_t)_{t \ge 0}$ is the process on~$\Z^d$ that satisfies~$Y^x_0 = x$ and
\begin{itemize}
    \item if~$Y^x_{t-}=y$ and~$t \in \mathcal J^{(y,z)}$, then we set~$Y^x_t = z$ (this is called an \emph{autonomous jump});
    \item if~$Y^x_{t-}=y$ and~$t \in \mathcal H^{\{y,z\}}$, then we set~$Y^x_t = z$ (this is called a \emph{stirring jump}).
\end{itemize}
Note that for any~$x_1,\dots,x_n \in \Z^d$, the Markov chain $(Y^{x_1}_t,\dots, Y^{x_n}_t)$ is a system of coalescing-stirring particles having~$\mathcal L^{(\mathsf v)}_\mathrm{cs}$ from~\eqref{eq_lvcs} as a generator. 

Analogously to what was described earlier (in the paragraphs before Lemma~\ref{familia_mu_alpha_properties}), we sample from~$\mu_{\alpha,\mathsf v}$ by using a random partition of~$\Z^d$, as well as some extra randomness. We fix an enumeration of~$\Z^d = \{z_1,z_2,\dots\}$ and write
\begin{equation}\label{eq_ordering}
    z_n < z_m \quad \text{if} \quad n < m.    
\end{equation}
We also define an i.i.d. sequence of Bernoulli random variables~$\left(\mathcal B(z_n)\right)_{n\ge 1}$ with parameter~$\alpha$. For any fixed~$\mathsf v \ge 0$ we define a random partition~$\mathscr C$ of~$\Z^d$ by declaring that~$x$ and $y$ belong to the same block if there exists~$t$ such that~$Y^x_t = Y^y_t$. Finally, for all $x \in \Z^d $ we set
\begin{equation}\label{def_xi_alpha}
    \xi^{(\alpha)}(x) = \mathcal B(z_{n(x)}),
\end{equation}
where~$n(x)$ is the smallest integer~$m$ such that~$z_m$ and~$x$ belong to the same block, that is, 
$$ n(x) := \inf\{ m \in \N : Y^{z_m}_t = Y^x_t \text{ for some } t \ge 0 \}. $$
It is easy to check that~$\xi^{(\alpha)}$ is distributed as $\mu_{\alpha,\mathsf v}$. 

We define the set-valued process of the locations of the coalescing-stirring particles of $A$ by
$$ S_t(A) := \{Y^x_t : x \in A\} \quad t \ge 0, $$
and~$|S_\infty(A)| := \lim_{t \to \infty}|S_t(A)|$, which exists for any finite~$A \subset \Z^d$ since~$|S_t(A)|$ is non-increasing. By the duality relationship~\eqref{eq_duality} we have that
\begin{equation}\label{eq_duality_graph_constr}
    \mu_{\alpha,\mathsf v}( \xi : \xi \equiv 1 \text{ on } A ) = \E[ \alpha^{|S_\infty(A)|} ] \quad \text{ for any } A \subset \Z^d \text{ finite.} 
\end{equation}

The next lemma can be proved the same way as Lemma~5.1 of \cite{RathValesin2017} with the following minor modification. The proof of Lemma~5.1 relies on an annihilation approach in which two particles annihilate as soon as they meet. In our case, two particles annihilate as soon as they are nearest neighbours (that is, when their $\ell^1$-distance is 1). When more than two particles become nearest neighbors, we annihilate the smallest neighbouring pair under the ordering~$<$ of~\eqref{eq_ordering}. The rest of the proof remains unchanged so we omit it.
\begin{lemma}\label{lemma_annihilation}
For any $A \subset \Z^d$ finite,~$\beta \in (0,1)$ and~$\mathsf v \ge 0$ we have
$$\E[ \beta^{|S_\infty(A)|} ] \leq \beta^{|A|} \cdot \prod_{ \substack{ x,y \in A \\ x<y } } \big( 1 + h(x,y) \cdot (\beta^{-2} - 1) \big),$$
where $h(x,y) = \Prob( \exists t \ge 0: |Y^x_t - Y^y_t|_1 = 1 )$. 
\end{lemma}
\begin{remark}\label{remark_value_h_c_1}
    Note that $h(x,y)$ is independent of~$\mathsf v$ and that, by a Green function estimate (see for instance~\cite{Lawler2010}),~$h(x,y) \leq c_\mathrm{nn} \cdot |x-y|^{2-d}_1$, where~$c_\mathrm{nn} > 0$ is a dimension-dependent constant.
\end{remark}

With this lemma at hand, we begin the comparison between~$\mu_{\alpha,\mathsf v}$ and~$\pi_\alpha$. Recall that~$\widehat{\alpha} = \min\{\alpha, 1-\alpha \}$.

\begin{lemma}\label{lemma_annihilation_2}
Let~$\Lambda \subset \Z^d$ be finite. Then, for any~$E \in \mathcal G_\Lambda$,~$\alpha \in (0,1)$ and~$\mathsf v \ge 0$ we have
$$\mu_{\alpha,\mathsf v}( E ) \leq \pi_\alpha(E) \cdot \prod_{ \substack{ x,y \in \Lambda \\ x<y } } \big( 1 + h(x,y) \cdot (\widehat{\alpha}^{-2} - 1) \big).$$ 
\end{lemma}
\begin{proof}
    Let~$A,B \subseteq \Lambda$ such that they form a partition of~$\Lambda$. Using the graphical construction introduced earlier and the random configuration ~$\xi^{(\alpha)}\sim \mu_{\alpha,\mathsf v}$ from~\eqref{def_xi_alpha}, we have that
    \begin{align*}
        \mu_{\alpha,\mathsf v}(\xi: \xi \equiv 1 \text{ on } A, \; \xi \equiv 0 \text{ on } B) & =  \Prob( \xi^{(\alpha)} \equiv 1 \text{ on } A, \; \xi^{(\alpha)} \equiv 0 \text{ on } B ) \\[0.2cm]
        & = \E[ \alpha^{|S_\infty(A)|} \cdot (1-\alpha)^{|S_\infty(B)|} \cdot \mathds{1}_\mathcal{D} ],    
    \end{align*}
    where~$\mathcal D$ is the event that the particles~$\{(Y^x_t)_{t \ge 0}: x \in A \}$ do not coalesce with the particles~$\{(Y^x_t)_{t \ge 0}: x \in B \}$, although they may exchange positions. Hence,
    \begin{align*}
        \E[ \alpha^{|S_\infty(A)|} \cdot (1-\alpha)^{|S_\infty(B)|} \cdot \mathds{1}_\mathcal{D} ] & \leq \alpha^{|A|} \cdot (1-\alpha)^{|B|} \cdot \E \Big[ \Big(\frac{1}{\widehat{\alpha}} \Big)^{|A| + |B| -|S_\infty(A)| - |S_\infty(B)|} \cdot \mathds{1}_\mathcal{D} \Big] \\[0.2cm]
        & \leq \alpha^{|A|} \cdot (1-\alpha)^{|B|} \cdot \E \Big[ \Big( \frac{1}{\widehat{\alpha}} \Big)^{|A \cup B|-|S_\infty(A \cup B)|} \Big] \\[0.2cm]
        & \leq \alpha^{|A|} \cdot (1-\alpha)^{|B|} \cdot \prod_{ \substack{ x,y \in \Lambda \\ x<y } } \big( 1 + h(x,y) \cdot (\widehat{\alpha}^{-2} - 1) \big), 
    \end{align*}
where the last inequality follows from Lemma~\ref{lemma_annihilation} with~$\beta = \widehat{\alpha}$. Then, the statement of the lemma readily follows. 
\end{proof}

Given a set~$\Lambda \subseteq \Z^d$ and a fixed configuration~$\eta \in \{0,1\}^\Lambda$, for every~$\zeta \in \{0,1\}^{\Z^d}$ we define the configuration~$\zeta \vee \eta \in \{0,1\}^{\Z^d}$ by
$$ (\zeta \vee \eta )(x) := \left\{ \begin{array}{ll}
    \zeta(x) & x \notin \Lambda \\[0.2cm]
    \max\{\zeta(x),\eta(x)\} & x \in \Lambda.
\end{array} \right. $$
We define~$\zeta \wedge \eta$ analogously. In the space of the graphical construction we define the $\sigma$-algebra~$\mathcal F_t := \sigma\{ (Y^x_s)_{s \in [0,t]}: x \in \Z^d \}$ for any~$ t \ge 0$. 

\begin{lemma}\label{lemma_mu_stir_pi_alpha}
    Let~$\Lambda \subset \Z^d$ be finite and~$E,F \in \mathcal G_\Lambda$ be an increasing and a decreasing event, respectively. In the space of the graphical construction, we consider an~$\mathcal F_t$-measurable random configuration~$\eta \in \{0,1\}^\Lambda$ with the property that the coalescing random walks started at~$\{x \in \Lambda: \eta(x) = 0\}$ do not coalesce until time~$t$. Then for any~$\alpha \in (0,1)$ and~$\mathsf v \ge 0$
    \begin{align*}
        & \mu_{\alpha,\mathsf v}(E) \leq \E \Big[ \pi_\alpha(\zeta: \zeta \vee \eta \in E) \cdot \prod_{ \substack{ x,y \in \Lambda \setminus \Xi \\ x< y} } \big( 1+ h(Y^x_t,Y^y_t) \cdot (\widehat{\alpha}^{-2}-1) \big) \Big], \\[0.2cm]
        & \mu_{\alpha,\mathsf v}(F) \leq \E \Big[ \pi_\alpha(\zeta: \zeta \wedge (1-\eta) \in F) \cdot \prod_{ \substack{ x,y \in \Lambda \setminus \Xi \\ x< y} } \big( 1+ h(Y^x_t,Y^y_t) \cdot (\widehat{\alpha}^{-2}-1) \big) \Big].
    \end{align*}
    where~$\Xi = \Xi(\eta) = \{x \in \Lambda : \eta(x) = 1 \}$.
\end{lemma}
\begin{proof}
We only prove the first inequality, as the second one is treated in the same way. We work on a probability space where a system of coalescing-stirring particles~$\{Y^x_s: x \in \mathbb Z^d,\; 0 \le s \le t\}$ is defined from Poisson instructions as explained earlier, and let~$(\mathcal F_s)_{0 \le s \le t}$ be the natural filtration of these Poisson processes. As in the statement, we let~$\eta \in \{0,1\}^\Lambda$ be~$\mathcal F_t$-measurable. We abbreviate
\[
\mathcal Y(x):=Y^x_t,\qquad x \in \Lambda.
\]

We augment the probability space by defining~$\xi \sim \mu_{\alpha,\mathsf v}$ independent of the Poisson processes. It follows from duality and stationarity that~$\xi \circ \mathcal Y$ (defined as the element of~$\{0,1\}^{\mathbb Z^d}$ given by~$x \mapsto \xi(\mathcal Y(x))$ has law~$\mu_{\alpha,\mathsf v}$. We can then write
\[
\mu_{\alpha,\mathsf v}(E) = \mathbb P(\xi \circ \mathcal Y \in E) = \mathbb E[ \mathbb P(\xi \circ \mathcal Y \in E  \mid \mathcal F_t)].
\]
Since~$E$ is increasing, we can bound
\[
\mathbb P(\xi \circ \mathcal Y \in E  \mid \mathcal F_t) \le \mathbb P( (\xi \circ \mathcal Y) \vee \eta \in E \mid \mathcal F_t).
\]
We write the right-hand side above as~$\mathbb P( \xi \in E' \mid \mathcal F_t)$, where~$E'$ is the set of values of~$\xi$ for which~$(\xi \circ \mathcal Y) \vee \eta \in E$ (note that~$E'$ is a random set, as it depends on~$\mathcal Y$ and~$\eta$, but since we are conditioning on~$\mathcal F_t$, we treat it as deterministic).
We now augment the probability space further by letting~$\zeta$ be random element of~$\{0,1\}^{\mathbb Z^d}$ distributed as~$\pi_\alpha$. By Lemma~\ref{lemma_annihilation_2} we have
\[
\mathbb P(\xi \in E' \mid \mathcal F_t) \le \mathbb P(\zeta \in E' \mid \mathcal F_t) \cdot \prod_{\substack{x,y \in \mathcal Y(\Lambda \backslash \Xi)\\ x<y}} \left(1+h(Y^x_t,Y^y_t)\cdot (\widehat\alpha^{-2}-1)\right).
\]
Finally, using the fact that~$\mathcal Y$ is one-to-one on~$\Lambda \backslash \Xi$, and~$\zeta$ is independent of~$\mathcal F_t$, it is easy to see that
\[\mathbb P(\zeta \in E' \mid \mathcal F_t) = \mathbb P((\zeta \circ \mathcal Y) \vee \eta \in E \mid \mathcal F_t) = \pi_\alpha(\zeta: \zeta \vee \eta \in E).\]

\end{proof}

\subsection{Translation of increasing and decreasing events}\label{ss_translation_inc_dec}
Throughout this section, we fix~$N \in \N$ and a proper embedding~$\mathcal T \in \mathcal P_N$. Recall the sublattices in~\eqref{scales_mathcal_L}. Note that the renormalization scheme we are adopting is based on random configurations defined on boxes of scale~$L$ centered at the leaves~$\{ \mathcal T(m): m \in T_{(N)} \}$.

We now fix~$L, \ell \in \N$ with~$\ell \ge 6$,~$\mathsf v \in [0,\infty)$ and~$T \in (0,\infty)$. In the space of the graphical construction we define the random configurations~$\eta^m_1, \eta^m_2 \in \{0,1\}^{B(\mathcal T(m),2L)}$ for every leaf~$m \in T_{(N)}$ by
\begin{align*}
    & \eta^m_1(x) := \mathds{1} \big\{ \exists y \in B(\mathcal T(m),2L):  \max_{0 \leq t \leq T}|Y^y_t - y| \ge \ell/4  \\
    & \hspace{5cm} \text{ or } (Y^y_t)_{0 \leq t \leq T} \text{ performs at least one autonomous jump} \big\}, \\[0.2cm]
    & \eta^m_2(x) := \mathds{1}{\{ \exists y \in B(\mathcal T(m), 2L): y < x, \; |Y^x_T - Y^y_T| \leq \ell^{1/4} \}}, \quad\quad x \in B(\mathcal T(m), 2L).
\end{align*}
In words,~$\eta^m_1$ is a configuration that is identically 1 or 0 on the box~$B(\mathcal T(m), 2L)$. This depends on whether, among the coalescing-stirring particles that start in the ball~$B(\mathcal T(m),2L)$, there is at least one that either ``moves too much'' before time~$T$ (in the sense that it reaches a distance at least~$\ell/4$ away from where it started), or that performs an autonomous jump before time~$T$.
The configuration~$\eta^m_2$ is obtained by inspecting the positions~$\{Y^{x}_T: x\in B(\mathcal T(m), 2L) \}$ one by one in the order given by the ordering~$<$ of~\eqref{eq_ordering} and assigning state~1 to a site~$z \in B(\mathcal T(m), 2L)$ whenever~$Y^z_T$ is ``too close'' (at distance~$\leq\ell^{1/4}$) to a previously inspected position.

Letting
\begin{equation}\label{Lambda_def}
    \Lambda := \bigcup_{m \in T_{(N)}} B(\mathcal T(m),2L),
\end{equation}
we also define the random configurations~$\xi_1,\xi_2 \in \{0,1\}^\Lambda$ by
\begin{equation}\label{xi_1_x_2_def}
\xi_1|_{B(\mathcal T(m),2L )}  = \eta^m_1,\quad \xi_2|_{B(\mathcal T(m),2L )} =  \eta^m_2 \quad \text{for every } m \in T_{(N)}.    
\end{equation}
Note that these configurations are well defined as the sets~$\{B(\mathcal T(m),2L)$,~$m \in T_{(N)}\}$ are disjoint by~\eqref{spread_out_all_scales}. 

In the proof of Lemma~\ref{lemma_decomposition_main} we will need to bound terms of the form
\begin{equation}\label{eq_mu_incr_and_decr}
    \mu_{\alpha,\mathsf v} \Big( \bigcap_{m \in T_{(N)}} \mathcal \theta_{\mathcal T(m)}E \Big) \quad \text{and} \quad \mu_{\alpha,\mathsf v} \Big( \bigcap_{m \in T_{(N)}} \mathcal \theta_{\mathcal T(m)}F \Big),
\end{equation}
where~$E,F \in \mathcal G_{B(2L)} $ is an increasing and a decreasing event, respectively.

Observe that~$\xi_1,\xi_2 \in \{0,1\}^\Lambda$ are~$\mathcal F_T$-measurable random configurations so we can use Lemma~\ref{lemma_mu_stir_pi_alpha} with~$\eta=\xi_1\vee\xi_2$ and
\begin{equation}\label{Xi_def_xi12}
    \Xi = \{ x \in \Lambda: \xi_1(x) = 1 \text{ or } \xi_2(x) = 1 \}.
\end{equation}
to obtain bounds on~\eqref{eq_mu_incr_and_decr}. We can further bound these expressions using the following result. Recall the definition of~$c(\ell,L,\alpha,N)$ from~\eqref{c_ell_L_0_alpha_def}.

\begin{lemma}\label{lemma_product_deterministic_limit}
    Let~$\Lambda$ and~$\Xi$ be as in~\eqref{Lambda_def} and~\eqref{Xi_def_xi12}, respectively. Then, for any~$\alpha \in (0,1)$ and~$\mathsf v \ge 0$ we have that
    $$ \prod_{ \substack{ x,y \in \Lambda \setminus \Xi \\ x< y} } \big( 1+ h(Y^x_T,Y^y_T) \cdot (\widehat{\alpha}^{-2}-1) \big) \leq c(\ell,L,\widehat{\alpha},N ) \quad \text{a.s.}$$
\end{lemma}
\begin{proof}
Using the bounds~$1+x \le e^x$ and~$h(x,y) \le c_{\mathrm{nn}} \cdot |x-y|^{2-d}$ from 
Remark~\ref{remark_value_h_c_1}, we have
\begin{equation}\label{cond_exp_F_T}
    \prod_{ \substack{ x,y \in \Lambda \setminus \Xi \\ x< y} } \big( 1+ h(Y^x_T,Y^y_T) \cdot (\widehat{\alpha}^{-2}-1) \big) \leq \prod_{ x \in \Lambda \setminus \Xi} \exp \big\{ c_\mathrm{nn} \cdot \widehat{\alpha}^{-2} \cdot \sum_{ \substack{y \in \Lambda \setminus \Xi \\ y>x} } |Y^{x}_T - Y^{y}_T|^{2-d} \big\}.
\end{equation}

We now work on bounding the sum inside the above expression; to do so, fix $x \in \Lambda \setminus \Xi$. We write
$$\mathcal{V}(x,k) := 
\{y \in \Lambda \setminus \Xi: \; |y-x| \le \ell^k L/2\}, \quad k \ge 1.
$$
Note that $\cup_{k\geq 1} \mathcal V(x, k)=\Lambda \setminus \Xi$. As a consequence of~\eqref{spread_out_all_scales} we have that for any~$k \ge 1$
\begin{equation}\label{ell_k_bound}
    |\mathcal V(x,k)| \stackrel{\eqref{Lambda_def}} \leq |B(2L)| \cdot | \{ m \in T_{(N)}: \mathrm{dist} \big(x, B(\mathcal T(m),2L) \big) \leq \ell^k  L/2 \}| \leq |B(2L)| \cdot 2^{k-1}.
\end{equation}
Then,
\begin{equation}\label{sum_Y_T_complete}
     \sum_{ \substack{ y \in \Lambda \setminus \Xi \\ y>x} } |Y^{x}_T - Y^{y}_T|^{2-d} \leq \sum_{ \substack{ y \in \mathcal V(x,1) \\ y > x } } |Y^{x}_T - Y^{y}_T|^{2-d} + \sum_{k\ge 1}\;\;\sum_{ \substack{ y \in \mathcal V(x,k+1) \backslash \mathcal V(x,k) \\ y > x } } |Y^{x}_T - Y^{y}_T|^{2-d}.
\end{equation}
It is easy to treat the first sum on the right-hand side. Recall that, by the definition of~$\Xi$ in~\eqref{Xi_def_xi12}, if~$y \notin \Xi$ and~$y > x$, then~$|Y^y_T - Y^x_T| \ge \ell^{1/4}$. Hence,
\begin{equation*}
    \sum_{ \substack{ y \in \mathcal V(x,1) \\ y > x } }  |Y^{x}_T - Y^{y}_T|^{2-d} \leq \ell^{(2-d)/4} \cdot |\mathcal V(x,1)| \stackrel{\eqref{ell_k_bound}}\leq |B(2L)| \cdot \ell^{(2-d)/4}.
\end{equation*}

To treat the remaining terms in~\eqref{sum_Y_T_complete} fix~$k \ge 1$ and let~$y \in \mathcal V(x,k+1) \backslash \mathcal V(x,k)$. Since~$x,y \notin \Xi$ the coalescing-stirring particles~$Y^x_t,Y^y_t$ move less than~$\ell/4$ before time~$T$, so we have that~$|Y^x_T - Y^y_T| \ge |x-y|-\ell/2$. Combining this with the fact that~$y > x$ and~$y\notin \mathcal V(x,k)$, we obtain
\begin{equation}\label{stir_distance}
    \big| Y^{x}_T - Y^{y}_T \big| \geq  \frac{\ell^k  L}{2} - \frac{\ell}{2} \ge \ell^k
\end{equation}
for~$L \ge 3$. Hence,
\[
\sum_{ \substack{ y \in \mathcal V(x,k+1) \backslash \mathcal V(x,k) \\ y > x }} |Y^{x}_T - Y^{y}_T|^{2-d} \le  |\mathcal V(x,k+1) \backslash \mathcal V(x,k)|\cdot \ell^{k(2-d)}.
\]
Also using the bound~$|\mathcal V(x,k+1) \backslash \mathcal V(x,k)| \le |\mathcal V(x,k+1)| \le |B(2L)| \cdot 2^k$ by~\eqref{ell_k_bound} and summing over~$k$ we obtain
\[
\sum_{k\ge 1}\;\;\sum_{ \substack{y \in \mathcal V(x,k+1) \backslash \mathcal V(x,k) \\ y > x}} |Y^{x}_T - Y^{y}_T|^{2-d} \le |B(2L)| \cdot \sum_{k\ge 1} (2\ell^{2-d})^k \le 2|B(2L)| \cdot \ell^{2-d} \cdot \sum_{k\ge 0} (2\ell^{2-d})^k.
\]
Plugging the bounds we have obtained back into~\eqref{sum_Y_T_complete}, we have proved that
\begin{equation}
    \begin{aligned}
    \sum_{ \substack{ y \in \Lambda \setminus \Xi \\ y>x} } |Y^{x}_T - Y^{y}_T|^{2-d} & \le |B(2L)| \cdot \ell^{(2-d)/4} + 2|B(2L)| \cdot \ell^{2-d} \cdot \sum_{k\ge 0} (2\ell^{2-d})^k \\[0.2cm]
    & \leq |B(2L)| \cdot \ell^{(2-d)/4} \cdot \Big( 1 + 2 \cdot \sum_{k \ge 0} (2 \ell^{2-d})^k \Big) \\[0.2cm]
    & \leq 5 \cdot |B(2L)| \cdot \ell^{(2-d)/4},
    \end{aligned}
\end{equation}
where the last inequality follows from the assumption~$\ell \ge 6$. In fact, the argument only requires that~$\ell \ge 4$. We conclude by putting this back into~\eqref{cond_exp_F_T} and noting that~$|\Lambda \setminus \Xi| \leq |\Lambda| = |B(2L)| \cdot 2^N$.
\end{proof}

Therefore, applying Lemma~\ref{lemma_mu_stir_pi_alpha} and Lemma~\ref{lemma_product_deterministic_limit} gives the following statement. Let~$E,F \in \mathcal G_{B(2L)} $ be an increasing and a decreasing event, respectively. Then
\begin{equation}\label{ineq_mu_stir_int_C_m}
    \begin{aligned}
        & \mu_{\alpha,\mathsf v} \Big( \bigcap_{m \in T_{(N)}} \mathcal \theta_{\mathcal T(m)}E \Big) \leq c(\ell, L,\widehat{\alpha},N) \cdot \E \Big[ \pi_\alpha(\zeta: \zeta \vee \xi_1 \vee \xi_2 \in \bigcap_{m \in T_{(N)}} \mathcal \theta_{\mathcal T(m)} E ) \Big],\\
        & \mu_{\alpha,\mathsf v} \Big( \bigcap_{m \in T_{(N)}} \mathcal \theta_{\mathcal T(m)}F \Big) \leq c(\ell, L,\widehat{\alpha},N) \cdot \E\Big[ \pi_\alpha(\zeta: \zeta \wedge (1-\xi_1) \wedge (1-\xi_2) \in \bigcap_{m \in T_{(N)}} \mathcal \theta_{\mathcal T(m)} F ) \Big].
    \end{aligned}
\end{equation}

\subsection{Decomposition and stirring particles}\label{ss_decomposition}
In this section, we focus on deriving upper bounds for the expected values appearing on the right-hand side of the inequalities in~\eqref{ineq_mu_stir_int_C_m}. We work in the general setting of Subsection~\ref{ss_translation_inc_dec}.

Given a non-empty increasing event~$ E \in \mathcal G_{B(2L)}$  and the random configurations $\xi_1$ and $\xi_2$ from~\eqref{xi_1_x_2_def} we start by doing the following decomposition: 
\begin{equation}\label{decomposition_1}
    \begin{aligned}
    \E \Big[  \pi_\alpha &\big( \zeta: \zeta \vee \xi_1 \vee \xi_2 \in \bigcap_{m \in T_{(N)}} \theta_{\mathcal T(m)} E \big) \Big] \\
    & = \sum_{A \subseteq T_{(N)}} \E \Big[ \prod_{m \in T_{(N)}\setminus A } \pi_\alpha( \zeta: \zeta \vee \eta^m_2 \in \theta_{\mathcal T(m)} E ) \cdot \mathds{1}\{ \eta^m_1 \equiv 0 \} \cdot \prod_{m \in A} \mathds{1}{ \{ \eta^m_1 \equiv 1 \} } \Big].    
    \end{aligned}
\end{equation}
The equality holds as~$E$ is a non-empty increasing event, meaning the all-1 configuration is included~$\theta_{\mathcal T(m)} E$ for every~$m \in T_{(N)}$.

Let us fix~$A \subseteq T_{(N)}$. From the graphical construction, we see that for each~$m \in T_{(N)}$, both the random variable~$\pi_\alpha( \zeta: \zeta \vee \eta^m_2 \in \theta_{\mathcal T(m)} E ) \cdot \mathds{1}\{ \eta^m_1 \equiv 0 \}$ and the event~$\{\eta^m_1 \equiv 1\}$ can be determined by the Poisson processes in the space-time set
$$ \{ z \in \Z^d: \mathrm{dist}(z,B(\mathcal T(m),2L) ) \leq 1 + \ell/4 \} \times [0,T], \quad m \in T_{(N)}. $$
As a consequence of~\eqref{spread_out_all_scales}, we see that these space-time sets are all disjoint and thus the random elements are independent. It follows that the right-hand side of~\eqref{decomposition_1} equals
$$ \sum_{A \subseteq T_{(N)}} \prod_{m \in T_{(N)}\setminus A } \E \big[ \pi_\alpha( \zeta: \zeta \vee \eta^m_2 \in \theta_{\mathcal T(m)} E ) \cdot \mathds{1}\{ \eta^m_1 \equiv 0 \} \big] \cdot \prod_{m \in A}  \Prob( \eta^m_1 \equiv 1). $$

We now fix~$m_0 \in T_{(N)}$ and define the random configurations~$\eta_1,\eta_2 \in \{0,1\}^{B(2L)}$ by~$\eta_1 = \theta_{-\mathcal T(m_0)} \eta_1^{m_0}$ and~$\eta_2 = \theta_{-\mathcal T(m_0)} \eta_2^{m_0}$ . By the translation invariance of the random variables we have that
\begin{equation}\label{eq_decomposition_result_1}
    \begin{aligned}
    \E \Big[  \pi_\alpha( & \zeta: \zeta \vee \xi_1 \vee \xi_2 \in \bigcap_{m \in T_{(N)}} \theta_{\mathcal T(m)} E ) \Big] \\
    & = \sum_{ A \subseteq T_{(N)} } \big( \E[ \pi_\alpha(\zeta: \zeta \vee \eta_2 \in E) \cdot \mathds{1}\{ \eta_1 \equiv 0 \} ] \big)^{|T_{(N)} \setminus A|} \cdot \big( \Prob(\eta_1 \equiv 1) \big)^{|A|}.  
    \end{aligned}
\end{equation}
Following the same steps analogously, we have that for any~$F \in \mathcal G_{B(2L)}$ decreasing event
\begin{equation}\label{eq_decomposition_result_2}
    \begin{aligned}
    \E \Big[ \pi_\alpha( \zeta: \zeta \wedge & (1-\xi_1) \wedge (1-\xi_2) \in \bigcap_{m \in T_{(N)}} \theta_{\mathcal T(m)} F ) \Big] \\
    & = \sum_{ A \subseteq T_{(N)} } \big( \E[ \pi_\alpha(\zeta: \zeta \wedge (1-\eta_2) \in F) \cdot \mathds{1}\{ \eta_1 \equiv 0 \} ] \big)^{|T_{(N)} \setminus A|} \cdot \big( \Prob(\eta_1 \equiv 1 ) \big)^{|A|}.
\end{aligned}    
\end{equation}

We now state two lemmas, Lemma~\ref{lemma_Psi_1} and Lemma~\ref{lemma_Psi_incr_decr}, to further bound the expressions~\eqref{eq_decomposition_result_1} and~\eqref{eq_decomposition_result_2}. In what follows we set
\begin{equation}\label{T_v0_def}
    T := \frac{\ell}{1+\mathsf{v}} \quad \text{ and } \quad \mathsf{v}_0 := \ell \exp\{\ell\}-1.
\end{equation}

\begin{lemma}\label{lemma_Psi_1}
For~$\mathsf v \ge \mathsf{v}_0$ we have
\begin{align}\label{eq_big_pars}
    \Prob(\eta_1 \equiv 1) \le |B(2L)| \cdot \left(2d \left(1+\frac{d}{4}\right)^{-\ell/8} + \exp\{-\ell\}\right).
\end{align}
\end{lemma}
We will need the following fact stated in Corollary~2.6 of~\cite{RathValesin2017}, which is a consequence of a useful martingale result (see Theorem~\ref{thm_kallemberg} in the Appendix).
\begin{proposition}
    Assume~$d \ge 1$. Let~$(X_t)_{t \ge 0 }$ be a continuous-time simple random walk on~$\Z^d$ of jump rate 1. Assume~$X_0 = 0$. Then, for any $S,r \ge 0$ we have
    \begin{equation}\label{coro_rw_mart_bound}
        \Prob \Big( \max_{0 \leq t \leq S}|X_t| > r \Big) \leq 2d \exp \left\{ - \frac{1}{2}r \ln \Big( 1 + \frac{d \cdot r}{S} \Big) \right\}.
    \end{equation}
\end{proposition}

\begin{proof}[Proof of Lemma~\ref{lemma_Psi_1}]
Recall that $(Y^x_t)_t$ moves as a continuous-time simple random walk with rate $1+\mathsf{v}$. Therefore, using~\eqref{coro_rw_mart_bound} we have for all~$x \in B(2L)$
    \begin{equation*}\begin{split}
    \mathbb P \left( \max_{0\le t \le T} |Y^x_t - x| > \frac{\ell}{4}\right) &\le 2d \exp \left\{ -\frac{1}{8}\ell \ln \left( 1 + \frac{d}{4T(1+\mathsf{v})}\ell \right) \right\}
    \stackrel{\eqref{T_v0_def}}{=} 2d \left( 1+\frac{d}{4}\right)^{-\ell/8}.
    \end{split}
\end{equation*}
Furthermore, as $(Y^x_t)_t$ performs autonomous jumps at rate 1 we have
\begin{equation*}
\begin{split}
   \Prob \Big(  \begin{array}{c}
        (Y^x_t)_{t \in [0,T]} \text{ performs at least} \\
        \text{  one autonomous jump}
    \end{array}  \Big) = 1-e^{-T} \le T \stackrel{\eqref{T_v0_def}}{=} \frac{\ell}{1+\mathsf v} \le \frac{\ell}{1+\mathsf v_0} \stackrel{\eqref{T_v0_def}}{=} \exp\{-\ell\},
\end{split}
\end{equation*}
where in the second inequality we used the assumption that~$\mathsf v \ge \mathsf v_0$.
Applying a union bound over all~$x \in B(2L)$ and using the above inequalities, we obtain~\eqref{eq_big_pars}.
\end{proof}

Recall the definition of~$q = q(\ell,L)$ from~\eqref{q_def}.

\begin{lemma}\label{lemma_Psi_incr_decr}
There exist dimension-dependent constants~$a,b>0$ such that the following holds. Fix~$\ell,L \in \N$ with~$\ell \ge 6$ and~$q(\ell,L) \in (0,1)$. Then, for all~$\mathsf v \ge \mathsf v_0, \alpha \in (0,1)$, all increasing events~$E$ and decreasing events~$F$ in~$\mathcal G_{B(2L)}$, we have
    \begin{align*}
    & \E[ \pi_\alpha(\zeta: \zeta \vee \eta_2 \in E) \cdot \mathds{1}\{ \eta_1 \equiv 0 \} ] \leq \pi_{\alpha + q-\alpha q}(E) + a \cdot |B(2L)|^2\exp \Big\{ -\frac{b \cdot \ell^{1/4}}{|B(2L)|} \Big\}, \\
    & \E[ \pi_\alpha(\zeta: \zeta \wedge (1-\eta_2) \in F) \cdot \mathds{1}\{ \eta_1 \equiv 0 \} ] \leq \pi_{\alpha (1-q)}(F) + a \cdot |B(2L)|^2\exp \Big\{ -\frac{b \cdot \ell^{1/4}}{|B(2L)|} \Big\}.
    \end{align*}
\end{lemma}
With this lemma in hand, we now prove Lemma~\ref{lemma_decomposition_main}, and then return to the proof of the lemma.
\begin{proof}[Proof of Lemma~\ref{lemma_decomposition_main}:]
The term~$\Psi_1$ come from Lemma~\ref{lemma_Psi_1}, the constants~$a,b>0$ and the terms~$\Psi_\mathrm{incr}(E)$ and~$\Psi_\mathrm{decr}(F)$ come from Lemma~\ref{lemma_Psi_incr_decr}, and the value~$\mathsf v_0$ from~\eqref{T_v0_def}. We plug them back in~\eqref{eq_decomposition_result_1} and~\eqref{eq_decomposition_result_2}. We conclude by putting together the resulting inequality and~\eqref{ineq_mu_stir_int_C_m}.
\end{proof}

To prove Lemma~\ref{lemma_Psi_incr_decr} we introduce a system of particles that only perform stirring jumps. Given~$n \in \N$, a \emph{system of~$n$ stirring particles}, denoted by~$(W^1_t,\dots, W^n_t)_{t \ge 0}$, is the Markov chain with state space~$(\Z^d)^n$ and generator given by~$\mathcal L^{(1)}_\mathrm{stir}$ defined in~\eqref{generator_L_v_stirring}.

Note that on~$\{\eta_1 \equiv 0\}$ in the system of coalescing-stirring particles~$\{(Y^x_t)_{0 \leq t \leq T}: x \in B(2L)\}$ the particles only perform stirring jumps at rate~$\mathsf v$. On the probability space~$\mathbf{P}$, we define a system of stirring particles~$\{(W^x_t)_{0 \leq t \leq \mathsf v \cdot T} : x \in B(2L)\}$ such that~$W^x_0 = x$ for all~$x \in B(2L)$. Additionally, we define the random configuration~$\eta^\mathrm{stir}_2$ by
$$ \eta^\mathrm{stir}_2(x) := \mathds{1}\{ \exists y \in B(2L): y < x, \; |W^x_{\mathsf v \cdot T} - W^y_{\mathsf v \cdot T}| \leq \ell^{1/4} \}, \quad x \in B(2L) $$
It follows that
\begin{equation}\label{coupling_pure_stir}
    \begin{aligned}
    & \E[ \pi_\alpha(\zeta: \zeta \vee \eta_2 \in E) \cdot \mathds{1}\{ \eta_1 \equiv 0 \} ] \leq \mathbf{E}[\pi_\alpha(\zeta: \zeta \vee \eta^\mathrm{stir}_2 \in E)], \\[0.2cm]
   & \E[ \pi_\alpha(\zeta: \zeta \wedge (1-\eta_2) \in F) \cdot \mathds{1}\{ \eta_1 \equiv 0 \} ] \leq \mathbf{E}[\pi_\alpha(\zeta: \zeta \wedge (1-\eta^\mathrm{stir}_2) \in F)],        
    \end{aligned}
\end{equation}
where~$\mathbf{E}$ is the expectation operator associated with~$\mathbf{P}$.

We now use the following result to compare the location of the stirring particles with independent random walks. The proof is given in the Appendix.
\begin{proposition}\label{prop_couple_irw_ep}
Assume~$d \ge 3$. There exist dimension-dependent constants~$a,b > 0$ such that the following holds. Given~$k \in \N$, there exists a coupling under a probability measure~$\widehat{\Prob}$ of a system of stirring particles~$(\widehat{W}^1_t,\dots,\widehat{W}^k_t)_{t \ge 0}$ with a system of independent simple random walks~$(\widehat{X}^1_t,\dots,\widehat{X}^k_t)_{t \ge 0}$, each of them of jumping rate~1 and~$\widehat{W}^i_0 = \widehat{X}^i_0 = x_i$, for~$i=1,\dots,k$, such that for any~$r > 0$
    \begin{equation}\label{coupling_exp_A_B}
        \sup_{x_1,\dots,x_k \in \Z^d}\sup_{t \ge 0} \; \widehat{\Prob} \Big(\sup_{1 \leq i \leq k} \big|\widehat{W}^i_t - \widehat{X}^i_t \big| > r \Big) \leq a \cdot k^2 \exp\{-br/k\}.
    \end{equation}
\end{proposition}

We apply Proposition~\ref{prop_couple_irw_ep} with the particles initially located in the ball~$B(2L)$, that is~$k=|B(2L)|$ and with~$r = \ell^{1/4}$. The particles are labeled according to the well-ordering~$<$ from~\eqref{eq_ordering}. We denote~$\mathcal A:=\{\sup_{1 \leq i \leq k} \big|\widehat{W}^i_{\mathsf v \cdot T} - \widehat{X}^i_{\mathsf v \cdot T} \big| > \ell^{1/4}\}$ and define the random configuration~$\widehat{\eta^\mathrm{stir}_2}\in\{0,1\}^{B(2L)}$ (analogously to $\eta^\mathrm{stir}_2$ above) as
$$ \widehat{\eta^\mathrm{stir}_2}(x) := \mathds{1} \big\{ \exists y \in B(2L): y < x, |\widehat{W}^x_{\mathsf{v} \cdot T} - \widehat{W}^y_{\mathsf v \cdot T}| \leq \ell^{1/4} \big\}, \quad x \in B(2L).$$
On the event~$\mathcal A^c$ we have that if~$y< x$ and~$|\widehat{W}^x_{ \mathsf v \cdot T} - \widehat{W}^y_{\mathsf v \cdot T}| \leq \ell^{1/4}$, then
$$ \big| \widehat{X}^x_{\mathsf v \cdot  T}-\widehat{X}^y_{ \mathsf v \cdot T} \big| \leq \big| \widehat{W}^x_{\mathsf v \cdot T}-\widehat{W}^y_{\mathsf v \cdot T} \big| + 2\ell^{1/4} \leq 3\ell^{1/4}.$$
Hence, defining the random configuration~$\eta^\mathrm{rw}_2\in\{0,1\}^{B(2L)}$ as
\begin{equation}\label{eta_rw_def}
    \eta^\mathrm{rw}_2(x) := \mathds{1} \big\{ \exists y \in B(2L): y < x, |\widehat{X}^x_{\mathsf v \cdot T} - \widehat{X}^y_{\mathsf v \cdot T}| \leq 3\ell^{1/4} \big\}, \quad x \in B(2L),
\end{equation}
we have that~$\widehat{\eta^\mathrm{stir}_2}(x) \leq \eta^\mathrm{rw}_2(x)$ on the event $\mathcal A^c$ for every~$x \in B(2L)$. As a consequence, denoting by $\widehat{\E}$ the expectation operator associated with $\widehat{\Prob}$, we have
\begin{equation}\label{eq_binomial_to_ind_rw_incr}
    \begin{aligned}
        \mathbf{E}[\pi_\alpha  (\zeta:  \zeta & \vee \eta^\mathrm{stir}_2 \in E)] = \widehat{\E}[\pi_\alpha(\zeta: \zeta \vee \widehat{\eta^\mathrm{stir}_2} \in E)] \\[0.2cm]
        & \leq \widehat{\Prob}(\mathcal A) + \widehat{\E} \Big[\pi_\alpha(\zeta: \zeta \vee \widehat{\eta^\mathrm{stir}_2} \in E) \cdot \mathds{1}_{\mathcal{A}^c} \Big] \\[0.2cm]
        & \leq a \cdot |B(2L)|^2 \cdot \exp \big\{-b \cdot \ell^{1/4}/|B(2L)| \big\} + \widehat{\E} \Big[\pi_\alpha(\zeta: \zeta \vee \eta^\mathrm{rw}_2 \in E)\Big],
    \end{aligned}
\end{equation}
where the last inequality follows by~$E$ being an increasing event. Analogously,
\begin{equation}\label{eq_binomial_to_ind_rw_decr}
    \begin{aligned}
        \mathbf{E}[\pi_\alpha  (\zeta: \zeta & \wedge (1-\eta^\mathrm{stir}_2) \in F)] \\
        & \leq a \cdot |B(2L)|^2 \cdot \exp \big\{-b \cdot \ell^{1/4}/|B(2L)| \big\} + \widehat{\E} \Big[\pi_\alpha(\zeta: \zeta \wedge (1-\eta^\mathrm{rw}_2) \in F)\Big].
    \end{aligned}
\end{equation}

We now show that the random configuration~$\eta^\mathrm{rw}_2\in\{0,1\}^{B(2L)}$ can be stochastically dominated by a configuration sampled from a product Bernoulli distribution on~$B(2L)$. 
\begin{lemma}\label{lemma_stochastic_domination}
    There is a dimension-dependent constant~$c_d> 0$ such that the following holds. Given~$k \in \N$, let~$x_1,\dots,x_k \in \Z^d$ and let~$(X^1_t,\dots,X^k_t)_{t \ge 0}$ be a system of independent random walks started from~$(x_1,\dots,x_k)$. For fixed~$r > 0$, we write
    \begin{equation}\label{eta_def_lemma_bernoulli}
        \eta(x_n):= \mathds{1} \{ X^n_t \in B(X^1_t,r) \cup \cdots \cup B(X^{n-1}_t,r) \}, \quad n =1,\dots,k.
    \end{equation}
    Then, the random configuration~$\{\eta(x_n): n = 1,\dots,k\}$ is stochastically dominated by a product Bernoulli distribution on~$\{0,1\}^{\{x_1,\dots,x_k\}}$ with density parameter~$k \cdot c_d \cdot r^d \cdot t^{-d/2}$.
\end{lemma}
Before we prove Lemma~\ref{lemma_stochastic_domination}, let us see how it allows us to conclude. We will use the following results. Let~$r,s \in (0,1)$ and~$X,Y$ be independent random variables with distribution~$\mathrm{Ber}(r)$ and~$\mathrm{Ber}(s)$, respectively. Then
\begin{equation}\label{ber_p_q}
    X \vee Y \stackrel{\mathrm{d}}{=} \mathrm{Ber}(r+s-rs) \quad \text{and} \quad X \wedge Y \stackrel{\mathrm{d}}{=} \mathrm{Ber}(rs).
\end{equation}

\begin{proof}[Proof of Lemma~\ref{lemma_Psi_incr_decr}]
    We use Lemma~\ref{lemma_stochastic_domination} with the coalescing-stirring particles initially located in the ball~$B(2L)$, where~$k = |B(2L)|$. Moreover, we set~$t = \mathsf v \cdot T$ and~$r = 3\ell^{1/4}$. We label the particles according to the well-ordering~$<$ from~\eqref{eq_ordering}. Note that~$\eta(x_n)$ from~\eqref{eta_def_lemma_bernoulli} coincide with~$\eta^\mathrm{rw}(x_n)$ from~\eqref{eta_rw_def} for all~$n$. Then, using~\eqref{ber_p_q}, we have that
    \begin{equation}\label{E_hat_ineq}
        \begin{aligned}
        & \widehat{\E} \Big[\pi_\alpha(\zeta: \zeta \vee \eta^\mathrm{rw}_2 \in E)\Big] \leq \pi_{\alpha + q - \alpha q}(E)\\[0.2cm]
        & \widehat{\E} \Big[\pi_\alpha(\zeta: \zeta \wedge (1-\eta^\mathrm{rw}_2) \in F)\Big] \leq \pi_{\alpha (1-q) }(F) 
    \end{aligned}
    \end{equation}
    with density parameter
    \begin{equation}
        |B(2L)| \cdot c_d \cdot (3\ell^{1/4})^d \cdot (\mathsf v \cdot T)^{-d/2} \leq 2^{d/2}\cdot 3^d \cdot c_d \cdot |B(2L)| \cdot \ell^{-d/4} \stackrel{\eqref{q_def}}= q,    
    \end{equation}
where the inequality follows from using the definition of~$T$ in~\eqref{T_v0_def} and noting that~$\tfrac{1}{2} \leq \tfrac{\mathsf v}{1+ \mathsf v}$ for~$\mathsf v \ge \mathsf v_0 \ge 1$. We conclude by putting together \eqref{coupling_pure_stir}, \eqref{eq_binomial_to_ind_rw_incr}, \eqref{eq_binomial_to_ind_rw_decr} and~\eqref{E_hat_ineq}. 
\end{proof}

To prove Lemma~\ref{lemma_stochastic_domination} we recall the \emph{heat kernel bound} for the continuous-time random walk~$(X_t)_{t \ge 0}$, which is a consequence of the local central limit theorem: there exists dimension-dependent constants $C<\infty$ and $c>0$ such that for any~$x,y \in \Z^d$ we have
\begin{equation}\label{ineq_heat_kernel_rw}
     \Prob( X_t = y \mid X_0 = x ) \leq C \cdot t^{-\frac{d}{2}}\exp \left(-c \cdot \frac{|x-y|_1^2}{t} \right).
\end{equation}
\begin{proof}[Proof of Lemma~\ref{lemma_stochastic_domination}]
Set~$c_d:=\sup_{r>0}C \cdot |B(r)| \cdot r^{-d}$, which does not depend on~$r$. To obtain the desired stochastic domination it suffices to prove that  
$$ \Prob( \eta(x_n) = 1 \mid \eta(x_1),\dots,\eta(x_{n-1}) ) \leq k \cdot c_d \cdot r^d \cdot t^{-d/2} \quad \text{a.s.} \; \text{ 
for any } n =2,\dots,k  $$
Note that by definition~$\eta(x_1) = 0$. Fix~$n \ge 2$, we have that
\begin{equation}
    \begin{aligned}
        \Prob( \eta(x_n) = 1 \mid \eta(x_1),\dots,\eta(x_{n-1}) ) & = \Prob \big( X^n_{t} \in B \big(X^1_{t},r \big) \cup \cdots \cup B \big(X^{n-1}_{t},r \big) \mid  X^1_t,\dots,X^{n-1}_t \big) \\[0.2cm]
        & \leq  \sum_{i=1}^{n-1} \Prob( |X^n_t - X^i_t| \leq r \mid X^i_t) \\[0.2cm]
        & \stackrel{(\ast)} \leq (n-1) \cdot |B(r)| \cdot C \cdot t^{-d/2} \leq k \cdot c_d \cdot r^d \cdot t^{-d/2},
    \end{aligned}
\end{equation}
where~($\ast$) follows from the heat kernel bound~\eqref{ineq_heat_kernel_rw}, by upper bounding the exponential term by~$1$. 
\end{proof}

\section{Proof of Theorem~\ref{thm_main_2}: No percolation for $\alpha$ below~$p_c$ and~$\mathsf v$ large}\label{section_4}
Recall the event~$\mathrm{Perc}$ from~\eqref{Perc_def_set}. Given any~$\mathsf v \ge 0$, by property \textbf{(B)} of Lemma~\ref{familia_mu_alpha_properties} we have that all the elements of the family~$\{\mu_{\alpha, \mathsf v}: \alpha \in [0,1]\}$ are ergodic and then
$$ \mu_{\alpha,\mathsf v}(\mathrm{Perc}) \in \{0,1\} \quad \text{for any } \alpha \in [0,1].$$
Also, by property \textbf{(C)} of Lemma~\ref{familia_mu_alpha_properties} the family~$\{\mu_{\alpha,\mathsf v}: \alpha \in [0,1]\}$ is stochastically increasing in~$\alpha$, and since~$\mathrm{Perc}$ is an increasing event this implies the existence of a critical density~$\alpha_c = \alpha_c(\mathsf v) \in [0,1]$ such that~$\mu_{\alpha,\mathsf v}(\mathrm{Perc}) = 0$ for any~$\alpha < \alpha_c$ and~$\mu_{\alpha,\mathsf v}(\mathrm{Perc}) = 1$ for any~$\alpha > \alpha_c$. 
Recall that
$$ p_c = \sup\{p: \pi_p (\mathrm{Perc}) = 0 \} $$
is the critical density of classical Bernoulli site percolation on~$\Z^d$. We aim to prove that
$$ \lim_{\mathsf v \to \infty}\alpha_c(\mathsf v) = p_c. $$
In this section, we establish the lower bound
\begin{equation}\label{eq_part_1_thm_2}
    \liminf_{\mathsf v \to \infty} \alpha_c(\mathsf v) \geq p_c.
\end{equation}
The corresponding upper bound is established in Section~\ref{section_6}.

\begin{proof}[Proof of the lower bound:]
Fix an arbitrary~$\alpha < p_c$, we prove that there is a choice of~$L$ and~$\ell$ in~\eqref{scales_mathcal_L0} such that
\begin{equation}\label{eq_part_1_thm_perc}
    \begin{aligned}
    &\text{there exist } \mathsf v_0 > 0 \text{ such that for any } \mathsf v \ge \mathsf v_0,
    \\
    &\hspace{5cm}\mu_{\alpha,\mathsf v} \big( B(L_N - 2) \xleftrightarrow{\; \xi \;} B(2L_N)^c \big) \leq 2^{-2^N}, \quad N \in \N.    
    \end{aligned}
\end{equation}
As observed in~\eqref{eq_part_1_imply_noperc}, this implies that if~$\mathsf v \ge \mathsf v_0$ we have~$\mu_{\alpha,\mathsf v}(\mathrm{Perc}) = 0$. As the choice of~$\alpha<p_c$ is arbitrary, this further implies~\eqref{eq_part_1_thm_2}.

We now start the work to prove~\eqref{eq_part_1_thm_perc}. For any~$N \in \N$, applying~\eqref{path_anulus}, a union bound and using~\eqref{bound_for_Lambda_N} we have that
\begin{equation}\label{annulus_ineq_max_Lambda_pc}
    \begin{aligned}
     \mu_{\alpha,\mathsf v} \big(  B(L_N & - 2)  \xleftrightarrow{\; \xi \;} B(2L_N)^c \big) \\[2mm]
     & \leq (C_d \cdot \ell^{2d-2})^{2^N} \cdot \max_{\mathcal T \in \mathcal P_N } \mu_{\alpha,\mathsf v} \Big( \bigcap_{m \in T_{(N)}} \{ B(\mathcal T(m),L) \xleftrightarrow{\; \xi \;} B( \mathcal T(m),2 L)^c \} \Big).  
    \end{aligned}
\end{equation}

Fix~$N \in \N$ and a proper embedding~$\mathcal T \in \mathcal P_N$. We set
\begin{equation}\label{L_0_ell_def}
    \ell = L^{16d} \quad \text{and} \quad \mathsf v_0 = \ell \exp(\ell) -1 \; (\text{cf. } \eqref{T_v0_def}).
\end{equation}
With these choices of parameters we can choose $L$ large enough such that~\eqref{eq_part_1_thm_perc} holds in the following way.

Note that~$\{ B(L) \xleftrightarrow{\; \xi \;} B(2L)^c \} \in \mathcal G_{B(2L)}$ is an increasing event and $\{ B(\mathcal T(m),L) \xleftrightarrow{\; \xi \;} B( \mathcal T(m),2 L)^c \}= \theta_{ \mathcal T (m) }  \{ B(L) \xleftrightarrow{\; \xi \;} B(2L)^c \}$. Then, applying Lemma~\ref{lemma_decomposition_main} we have 
\begin{equation}\label{ineq_nontrivial_percolation_phase_transition}
    \begin{aligned}
    \mu_{\alpha,\mathsf v}  \Big(  \bigcap_{m \in T_{(N)}} & \theta_{ \mathcal T (m) }  \{ B(L) \xleftrightarrow{\; \xi \;} B(2L)^c \}  \Big)  \\
    & \leq c(\ell,L,\widehat{\alpha},N) \cdot \sum_{ A \subseteq T_{(N)} } \big(\Psi_\mathrm{incr}( \{ B(L) \xleftrightarrow{\; \xi \;} B(2L)^c \} )\big)^{|T_{(N)} \setminus A|} \cdot \Psi^{|A|}_1.   
    \end{aligned}
\end{equation}
We now bound all the above terms. There exists $\mathsf L_{(1)}$ such that for~$L \ge \mathsf L_{(1)}$, we have
\begin{equation}\label{c_upper_bound_no_percolation}
    c(\ell,L,\widehat{\alpha},N) \stackrel{\eqref{c_ell_L_0_alpha_def},\eqref{L_0_ell_def}} = \exp\{5c_\mathrm{nn} \cdot \widehat{\alpha}^{-2} |B(2L)|^2 \cdot 2^N \cdot L^{4d(2-d)} \} \leq \exp\{ 2^N \}.
\end{equation}
Recalling the definition of~$q$ from~\eqref{q_def}, we see that there exists~$\mathsf L_{(2)}$ such that for~$L \ge \mathsf L_{(2)}$
\begin{equation}\label{q_bound_incr}
    q \stackrel{\eqref{L_0_ell_def}}= 2^{d/2} \cdot 3^d \cdot c_d \cdot |B(2L)| \cdot L^{-4d^2} < (p_c - \alpha)/2.
\end{equation}
Then, using the fact that~$\pi_\alpha$ is stochastically increasing in~$\alpha$, we have for~$L \ge \mathsf L_{(2)}$
\begin{equation}\label{pi_incr_bound_1}
    \pi_{\alpha + q - \alpha q } \big( B(L) \xleftrightarrow[]{\; \xi \;} B(2L)^c \big) \stackrel{\eqref{q_bound_incr}}\leq \pi_{ \frac{\alpha + p_c }{2} } \big( B(L) \xleftrightarrow[]{\; \xi \;} B(2L)^c \big) \leq |B(L)| \cdot \pi_{\frac{\alpha + p_c}{2}}\big( 0 \xleftrightarrow[]{\; \xi \;} \partial B(L) \big),
\end{equation}
where~$\partial B(L) := B(L) \setminus B(L-1)$. Since~$\tfrac{\alpha + p_c}{2} < p_c$, there exists~$c = c(\alpha) > 0$ such that
\begin{equation}\label{pi_incr_bound_2}
    \pi_{\frac{\alpha + p_c}{2}}\big( 0 \xleftrightarrow[]{\; \xi \;} \partial B(L) \big) \leq \exp\{-c \cdot L\}.
\end{equation}
This is a consequence of the sharpness of the phase transition for Bernoulli site percolation (see for instance Theorem 1.1 of~\cite{DuminilTassion16}). Putting together~\eqref{L_0_ell_def},~\eqref{pi_incr_bound_1} and~\eqref{pi_incr_bound_2}, we see that there exists~$\mathsf L_{(3)} \ge \mathsf L_{(2)}$ such that for~$L \ge \mathsf L_{(3)}$
\begin{align*}
    &\Psi_\mathrm{incr}( \{ B(L) \xleftrightarrow{\; \xi \;} B(2L)^c \}) \stackrel{\eqref{Psi_incr_def}} \leq |B(L)| \cdot \exp \{ -c \cdot L\} + a \cdot |B(2L)|^2 \exp \Big\{ -\frac{b \cdot L^{4d}}{|B(2L)|} \Big\} \\[.2cm]
    & \hspace{4.78cm} \leq \exp \{ -c/2 \cdot L\},\\[.2cm]
    & \Psi_1 \stackrel{\eqref{Psi_1_def},\eqref{L_0_ell_def}} = |B(2L)| \cdot \Big( 2d \Big( 1 + \frac{d}{4} \Big)^{-L^{16d}/8} + \exp\{-L^{16d}\} \Big) \leq \exp \{ -c/2 \cdot L\}.
\end{align*}
We now bound~\eqref{annulus_ineq_max_Lambda_pc} using~\eqref{ineq_nontrivial_percolation_phase_transition},~\eqref{c_upper_bound_no_percolation} and the above two inequalities to obtain for~$L \ge \max\{\mathsf L_{(1)},\mathsf L_{(3)}\}$,
\begin{align*}
\mu_{\alpha,\mathsf v} \big(  B(L_N-2) & \xleftrightarrow[]{\; \xi \;} B(2L_N)^c \big) \\[.2cm]
& \leq (C_d \cdot L^{32d^2-32d})^{2^N} \cdot \exp \{ 2^N \} \cdot \exp\{ - 2^N c/2 \cdot L \} \\[0.2cm]
& = (C_d  \cdot L^{32d^2-32d} \cdot \exp\{ 1 - c/2 \cdot L \} \big)^{2^N} ,
\end{align*}
which proves~\eqref{eq_part_1_thm_perc} by taking~$L$ large enough.
\end{proof}

\section{Decorrelation inequality for monotone couplings}\label{section_5}
In Section~\ref{ss_proof_proposition}, we introduced a graphical construction based on coalescing-stirring particles for sampling the random configuration~$\xi^{(\alpha)}$ defined in~\eqref{def_xi_alpha}, for~$\alpha \in [0,1]$ and~$\mathsf v \ge 0$, whose distribution is given by~$\mu_{\alpha,\mathsf v}$ in~\eqref{eq_duality_graph_constr}. To sample all the configuration~$\{\xi^{(\alpha)}:\alpha \in [0,1]\}$ in the same space we consider the natural monotone coupling of Bernoulli random configurations~$\{\zeta^{(\alpha)}: \alpha \in [0,1]\}$, that is,~$\zeta^{(\alpha)} \sim \pi_\alpha$ and for~$\alpha \leq \alpha':$
$$\zeta^{(\alpha)}(x) \leq \zeta^{(\alpha')}(x), \; \forall x \in \Z^d. $$
We then use these random configurations in the definition of~$\xi^{(\alpha)}$ in~\eqref{def_xi_alpha} to obtain a monotone coupling of~$\{\xi^{(\alpha)}: \alpha \in [0,1]\}$.

Given~$\alpha, \beta \in [0,1]$, with~$\alpha > \beta$, we denote by~$\mu_{\alpha,\beta,\mathsf v}$ and~$\pi_{\alpha,\beta}$ the distribution of~$(\xi^{(\alpha)},\xi^{(\beta)})$ and~$(\zeta^{(\alpha)},\zeta^{(\beta)})$, respectively. In this section we extend the decorrelation inequalities presented in Lemma~\ref{lemma_decomposition_main} to the measure~$\mu_{\alpha,\beta,\mathsf v}$ and~$\pi_{\alpha,\beta}$. We first introduce some notation. 

Given~$\Lambda \subset \Z^d$ finite, we let~$\mathcal G^{(2)}_\Lambda$ to be the set of subsets of~$\{0,1\}^{\mathbb Z^d} \times \{0,1\}^{\mathbb Z^d}$ of the form
$$ E = \{ (\xi_1,\xi_2) \in \{0,1\}^{\Z^d} \times \{0,1\}^{\Z^d}: (\xi_1|_\Lambda,\xi_2|_\Lambda) \in E_\Lambda\},\; \text{ where } E_\Lambda \subseteq \{0,1\}^\Lambda \times \{0,1\}^\Lambda.$$
Given~$E \in \mathcal G^{(2)}_\Lambda$ and~$x \in \Z^d$, we write~$\theta_x E := \{ (\theta_x \xi_1,\theta_x \xi_2): (\xi_1,\xi_2) \in E \}$. We say that~$E$ is an~\emph{increasing-decreasing} event if~$[ (\xi_1,\xi_2) \in E,\; \xi_1(x) \leq \xi'_1(x),\; \xi'_2(x) \leq \xi_2(x) \; \forall x \in \Lambda ]$ implies~$(\xi'_1,\xi'_2) \in E$, and a~\emph{decreasing-increasing} event if~$[ (\xi_1,\xi_2) \in E,\; \xi_1(x) \geq \xi'_1(x),\; \xi'_2(x) \geq \xi_2(x) \; \forall x \in \Lambda ]$ implies~$(\xi'_1,\xi'_2) \in E$.

Recall the definition of~$c(\ell,L,\alpha,N)$ and~$q(\ell,L)$ from~\eqref{c_ell_L_0_alpha_def} and~\eqref{q_def}.

\begin{lemma}\label{lemma_no_same_E}
There exist dimension-dependent constants~$a,b>0$ such that the following holds. Fix~$N, \ell,L \in \N$ with~$\ell \ge 6$ and~$q(\ell,L) \in (0,1)$ and a proper embedding~$\mathcal T \in \mathcal P_N$. Then, there exist~$\mathsf v_0 = \mathsf v_0(\ell)$ such that, for all~$\mathsf v \ge \mathsf v_0$, all~$\alpha, \beta \in (0,1)$ with~$\alpha > \beta$, and all families~$\{E_m: m \in T_{(N)} \}$, where each~$E_m \in \mathcal{G}^{(2)}_{B(2L)} $ is either increasing-decreasing or decreasing-increasing, we have
$$ \mu_{\alpha,\beta,\mathsf v} \Big(\bigcap_{m \in T_{(N)}} \theta_{\mathcal T(m)}E_m \Big) \leq c(\ell,L,\widehat{\theta},N) \cdot \sum_{A \subseteq T_{(N)}} \left( \prod_{m \in T_{(N)} \setminus A} \Psi(E_m) \right) \cdot (\Psi_1)^{|A|}, $$
where~$\widehat\theta := \min\{ 1-\alpha, \alpha - \beta, \beta \}$,~$\Psi_1$ was defined in~\eqref{Psi_1_def} and
$$ \Psi(E_m) = \left\{ \begin{array}{cl}
    \Psi_\mathrm{incr-decr}(E_m), & \text{if } E_m \text{ is an increasing-decreasing event},  \\[2mm]
    \Psi_\mathrm{decr-incr}(E_m), & \text{if } E_m \text{ is a decreasing-increasing event}, 
\end{array} \right. $$
with
\begin{equation}\label{def_Psi_incrdecr} 
\begin{aligned}
& \Psi_\mathrm{incr-decr}(E_m) = \pi_{\alpha +q -\alpha q,\; \beta(1-q)}(E_m) + a \cdot |B(2L)|^2 \exp \Big\{ - \frac{b \cdot \ell^{1/4}}{|B(2L)|} \Big\}. \\[2mm]
& \Psi_\mathrm{decr-incr}(E_m) = \pi_{\alpha(1-q),\; \beta + q -\beta q}(E_m) + a \cdot |B(2L)|^2 \exp \Big\{ - \frac{b \cdot \ell^{1/4}}{|B(2L)|} \Big\}.     
\end{aligned}
\end{equation}
\end{lemma}

To prove this lemma, we proceed analogously to the case of~$\mu_\alpha$. We start with an analogous result for Lemma~\ref{lemma_mu_stir_pi_alpha}.

\begin{lemma}\label{lemma_mu_pi_ineq_Prod}
Let~$\Lambda \subset \Z^d$ be finite and let~$E,F \in \mathcal G^{(2)}_\Lambda$ be an increasing-decreasing event and a decreasing-increasing event, respectively. In the space of the graphical construction, we consider an~$\mathcal F_t$-measurable random configuration~$\eta \in \{0,1\}^\Lambda$ with the property that the coalescing random walks started at~$\{x \in \Lambda: \eta(x) = 0\}$ do not coalesce until time~$t$. Then for any~$\alpha,\beta \in (0,1)$ with~$\alpha > \beta$ and~$\mathsf v \ge 0$,
\begin{align*}
& \mu_{\alpha,\beta,\mathsf v}(E) \leq \E \Big[\pi_{\alpha,\beta}\big( (\zeta_1,\zeta_2): (\zeta_1\vee \eta , \zeta_2 \wedge (1-\eta) ) \in E \big)  \cdot \hspace{-2mm} \prod_{ \substack{ x,y \in \Lambda \setminus \Xi \\ x< y} } \big( 1+ h(Y^x_t,Y^y_t) \cdot (\widehat\theta^{-2}-1) \big) \Big], \\[2mm] 
& \mu_{\alpha,\beta,\mathsf v}(F) \leq \E \Big[\pi_{\alpha,\beta}\big( (\zeta_1,\zeta_2): (\zeta_1 \wedge (1-\eta), \zeta_2 \vee \eta ) \in E \big)  \cdot \hspace{-2mm} \prod_{ \substack{ x,y \in \Lambda \setminus \Xi \\ x< y} } \big( 1+ h(Y^x_t,Y^y_t) \cdot (\widehat\theta^{-2}-1) \big) \Big],
\end{align*}
where~$\widehat\theta := \min\{ 1-\alpha, \alpha - \beta, \beta \}$.
Moreover, assume that~$\Lambda' \subset \Z^d$ is a finite set disjoint from~$\Lambda$, and let~$F \in \mathcal G^{(2)}_{\Lambda'}$ be a decreasing-increasing event. Then,
\begin{equation}\label{lemma_desintegration_main_mon}
\begin{aligned}
\mu_{\alpha,\beta,\mathsf v} \left( E \cap F \right) \leq  \E\Big[ \pi_{\alpha,\beta} \big( (\zeta_1,\zeta_2) &: (\zeta_1\vee \eta , \zeta_2 \wedge (1-\eta) ) \in E, (\zeta_1 \wedge (1-\eta), \zeta_2 \vee \eta ) \in F \big) \\[2mm]
& \cdot \prod_{ \substack{ x,y \in (\Lambda \cup \Lambda') \setminus \Xi \\ x< y} } \left( 1+ h(Y^x_t,Y^y_t) \cdot (\widehat\theta^{-2}-1) \right) \Big].
\end{aligned}
\end{equation}
\end{lemma}
\begin{proof}
Let~$A,B,C \subseteq \Lambda$ be three sets forming a partition of~$\Lambda$. By the definition of~$\mu_{\alpha,\beta,\mathsf v}$ using a monotone coupling, we have
\begin{align*}
    &\mu_{\alpha,\beta,\mathsf v}\big((\xi_1,\xi_2):\; \xi_1 \equiv \xi_2 \equiv 1 \text{ on } A; \;\xi_1 \equiv \xi_2 \equiv 0 \text{ on } B;\; \xi_1 \equiv 1-\xi_2 \equiv 1 \text{ on }C\big)\\[.2cm]
    &= \E \big[ \beta^{|S_\infty(A)|}  \cdot (1-\alpha)^{|S_\infty(B)|}\cdot (\alpha - \beta)^{|S_\infty(C)|} \cdot \mathds{1}_{\mathcal D} \big],
\end{align*}
where~$\mathcal D$ is the event that no two random walkers started from distinct sets among~$A,B,C$ coalesce. The right-hand side above is smaller than
\[
\beta^{|A|} \cdot (\alpha-\beta)^{|C|} \cdot (1-\alpha)^{|B|} \cdot \E \big[ (\widehat\theta^{-1})^{|\Lambda| - |S_\infty(\Lambda)|} \big].
\]
Summing over all partitions~$(A,B,C)$ of~$\Lambda$ gives
\begin{equation}\label{eq_coupling_measure_decorr_mu_pi}
    \mu_{\alpha,\beta,\mathsf v}(E) \leq \pi_{\alpha,\beta}(E) \cdot \prod_{ \substack{ x,y \in \Lambda \\ x< y} } \big( 1+ h(x,y) \cdot (\widehat\theta^{-2}-1) \big) \quad \text{for any } E \in \mathcal G^{(2)}_\Lambda.
\end{equation}
The rest of the proof is the same as in the proof of Lemma~\ref{lemma_mu_stir_pi_alpha}, so we omit it.
\end{proof}

We now establish the main result of this section. 
\begin{proof}[Proof of Lemma~\ref{lemma_no_same_E}]
For each~$m \in T_{(N)}$, let us fix a decreasing-increasing or increasing-decreasing event~$E_m$ in~$\mathcal G^{(2)}_{B(2L)}$. Recall the definitions of~$\xi_1$ and~$\xi_2$ from~\eqref{xi_1_x_2_def}, and let~$\Lambda$ and~$\Xi$ be given by~\eqref{Lambda_def} and~\eqref{Xi_def_xi12}, respectively. Applying Lemma~\ref{lemma_product_deterministic_limit} and~\eqref{lemma_desintegration_main_mon} from Lemma~\ref{lemma_mu_pi_ineq_Prod} with~$\eta = \xi_1 \vee \xi_2$, we have
\begin{equation}\label{ineq_decomposition_mu_alphabeta_1}
\begin{aligned}
\mu_{\alpha,\beta,\mathsf v} \Big( & \bigcap_{m \in T_{(N)}}  \theta_{\mathcal T(m)} E_m \Big) \leq c(\ell,L,\widehat{\theta},N) \\[2mm]
& \cdot \E\Big[ \pi_{\alpha,\beta} \big( (\zeta_1,\zeta_2): (\zeta_1 \vee \xi_1 \vee \xi_2, \zeta_2\wedge (1-\xi_1) \wedge (1-\xi_2) \big) \in \bigcap_{ \substack{m \in T_{(N)} \\ E_m \text{ incr-decr}} } \theta_{\mathcal T(m)}E_m,  \\[2mm]
& \qquad \qquad (\zeta_1 \wedge (1-\xi_1) \wedge (1-\xi_2), \zeta_2 \vee \xi_1 \vee \xi_2) \in \bigcap_{ \substack{m \in T_{(N)} \\ E_m \text{ decr-incr}} } \theta_{\mathcal T(m)}E_m \big) \Big].         
\end{aligned}
\end{equation}

The events~$E_m$ in~$\mathcal G^{(2)}_{B(2L)}, m \in T_{(N)},$ depend on configurations in the ball~$B(2L)$, so we can still use the decomposition argument at the beginning of Section~\ref{ss_decomposition} to obtain that the expected value of the right-hand side of~\eqref{ineq_decomposition_mu_alphabeta_1} is smaller than
\begin{equation}\label{eq_decomposition_mu_alphabeta_2}
\begin{aligned}
& \sum_{A \subseteq T_{(N)}} \left( \prod_{ \substack{m \in T_{(N)} \setminus A \\ E_m \text{ incr-decr} } } \E \big[ \pi_{\alpha,\beta}\big( (\zeta_1,\zeta_2): (\zeta_1 \vee \eta_2, \zeta_2 \wedge (1-\eta_2) ) \in E_m \big) \cdot \mathds{1}\{\eta_1 \equiv 0 \} \big] \right.  \\[2mm]
& \cdot \left. \prod_{ \substack{m \in T_{(N)} \setminus A \\ E_m \text{ decr-incr} }  } \E \big[ \pi_{\alpha,\beta}\big( (\zeta_1,\zeta_2): (\zeta_1 \wedge (1-\eta_2), \zeta_2 \vee \eta_2 ) \in E_m \big) \cdot \mathds{1}\{\eta_1 \equiv 0 \} \big] \right) \cdot \big( \Prob(\eta_1 \equiv 1) \big)^{|A|}.
\end{aligned}
\end{equation}
Recall the coupling in~$\widehat{\Prob}$ of a system of stirring particles and a system of independent random walks of Proposition~\ref{prop_couple_irw_ep}. Following the same arguments and the monotonicity of the event~$E_m$, we obtain that the expected value in~\eqref{eq_decomposition_mu_alphabeta_2} is smaller than (cf. \eqref{coupling_pure_stir} and \eqref{eq_binomial_to_ind_rw_incr})
\begin{equation}\label{ineq_alpha_beta_final_1}
\begin{aligned}
& \widehat{\E}\big[ \pi_{\alpha,\beta}\big( (\zeta_1,\zeta_2): (\zeta_1 \vee \eta^\mathrm{rw}_2, \zeta_2\wedge (1-\eta^\mathrm{rw}_2) ) \in E_m \big) \big] + a \cdot |B(2L)|^2 \cdot \exp\{-b \cdot \ell^{1/4}/|B(2L)|\}, \\[2mm]
& \widehat{\E}\big[ \pi_{\alpha,\beta}\big( (\zeta_1,\zeta_2): (\zeta_1 \wedge (1-\eta^\mathrm{rw}_2), \zeta_2\vee \eta^\mathrm{rw}_2 ) \in E_m \big) \big] + a \cdot |B(2L)|^2 \cdot \exp\{-b \cdot \ell^{1/4}/|B(2L)|\},
\end{aligned}
\end{equation}
when~$E_m$ is increasing-decreasing and decreasing-increasing, respectively.

We proved in Lemma~\ref{lemma_stochastic_domination} that~$\eta^\mathrm{rw}_2$ is stochastically dominated by a product Bernoulli distribution on~$\{0,1\}^{B(2L)}$ with density parameter~$q$ defined in~\eqref{q_def}. Using~\eqref{ber_p_q} and the monotonicity of the event~$E_m$, we have that
\begin{equation}\label{ineq_alpha_beta_final_2}
\begin{aligned}
& \widehat{\E}\big[ \pi_{\alpha,\beta}\big( (\zeta_1,\zeta_2): (\zeta_1 \vee \eta^\mathrm{rw}_2, \zeta_2\wedge (1-\eta^\mathrm{rw}_2)  ) \in E_m \big) \big] \leq \pi_{\alpha+q-\alpha q,\;\beta(1-q) }(E_m), \\[2mm]
& \widehat{\E}\big[ \pi_{\alpha,\beta}\big( (\zeta_1,\zeta_2): (\zeta_1 \wedge (1-\eta^\mathrm{rw}_2), \zeta_2 \vee \eta^\mathrm{rw}_2 ) \in E_m \big) \big] \leq \pi_{\alpha(1-q),\;\beta+q-\beta q }(E_m),    
\end{aligned}
\end{equation}
when~$E_m$ is increasing-decreasing and decreasing-increasing, respectively.

Finally, the term~$\Psi_1$ comes from Lemma~\ref{lemma_Psi_1}, the constants~$a,b>0$ and the terms~$\Psi_\mathrm{incr-decr}(E_m)$ and~$\Psi_\mathrm{decr-incr}(E_m)$ come from~\eqref{ineq_alpha_beta_final_1} and~\eqref{ineq_alpha_beta_final_2}. We conclude by plugging them back in~\eqref{eq_decomposition_mu_alphabeta_2} and then in~\eqref{ineq_decomposition_mu_alphabeta_1}.
\end{proof}

\section{Proof of Theorem~\ref{thm_main_2}: Percolation for $\alpha$ above~$p_c$ and~$\mathsf v$ large}\label{section_6}
In this section we prove
$$ \limsup_{\mathsf v \to \infty} \alpha_c(\mathsf v) \leq p_c. $$
We start by doing a coarse graining of the lattice. To do so, we introduce some notation. Fixing~$x \in \Z^d$ and~$\xi \in \{0,1\}^{\Z^d}$ with~$\xi(x)=1$, the \emph{cluster} of~$x$ is $\{y\in \Z^d: x\xleftrightarrow[]{\xi}y\}$. For a set~$A \subseteq \Z^d$, the \emph{diameter} of~$A$ is defined as
$$ \mathrm{diam}(A) := \sup\{ |x-y|_1 : x,y \in A \}. $$

Given~$M \in \N$, we define the sets~$E_M^{(1)} \subset \{0,1\}^{\Z^d}$, $E_M^{(2)} \subset (\{0,1\}^{\Z^d})^2$ and~$E_M \subset \{0,1\}^{\Z^d}$ as follows:
\begin{align*}
& E^{(1)}_M := \left\{ \begin{array}{c}
    \xi \in \{0,1\}^{\Z^d}: \text{ there exists an  open cluster}  \\
    \text{in  $\xi \vert_{B(M)}$ that intersects all the faces of } B(M) 
\end{array} \right\}, \\[2mm] 
& E^{(2)}_M := \left\{ \begin{array}{c}
    (\xi_1,\xi_2) \in (\{0,1\}^{\Z^d})^2 : \xi_1 \vert_{B(M)} \ge \xi_2 \vert_{B(M)} \text{ and  all open clusters in }\xi_2 \vert_{B( M )} \text{ of } \\
    \text{diameter  greater than  or equal to } M \text{ are connected in } \xi_1 \vert_{B(M)} 
\end{array} \right\}, \\[2mm]
& E_M := \left\{ \begin{array}{c}
    \xi \in \{0,1\}^{\Z^d}: \xi \vert_{B( M )} \text{ has a unique open cluster of diameter greater than }\\ \text{or equal to } M, \text{ and this cluster intersects all the faces of } B(M) 
\end{array} \right\}.
\end{align*}

\begin{remark}\label{rmk_inclusion}
Note that if~$\xi \in E_M$, then~$\xi \in E^{(1)}_M$ and~$(\xi',\xi) \in E^{(2)}_M$ for any~$\xi' \in \{0,1\}^{\Z^d}$ such that~$\xi' \vert_{B(M)} \ge \xi \vert_{B(M)}$. In particular, $\big(E_M^{(1)}\big)^c\subset E_M^c$ and $$\big(E_M^{(2)}\big)^c\cap \big( \{0,1\}^{\Z^d} \times E_M \big) \subset\{(\xi',\xi) \in (\{0,1\}^{\Z^d})^2 : \xi' \vert_{B(M)} \ngeq \xi \vert_{B(M)} \}.$$
\end{remark}

The following result, which is Corollary 7.4 in~\cite{Cerf2000}, establishes the exponential decay of the probability of the set~$E_M$ under the Bernoulli product measure. Although originally stated for dimension 3, it applies equally well to any dimension higher than 3.

\begin{theorem}\label{thm_supercritical}
    Let~$p > p_c$. Then there exists~$\mathsf{a} = \mathsf{a}(p,d)>0$ and~$\mathsf{b} = \mathsf{b}(p,d)>0$ such that
    \begin{equation}\label{eq_pisztora}
        \pi_p( E^c_M ) \leq \mathsf a \cdot \exp\{- \mathsf b \cdot M\} \quad \text{for every } M \in \N.
    \end{equation}
\end{theorem}
\begin{remark}
     This result is a specific application of Theorem 3.1 from~\cite{Pisztora96}.
     This theorem was originally established for random cluster measures, but the reference indicates that it can also be adapted for site percolation.
\end{remark}

\begin{proof}[Proof of the upper bound:]
Fix an arbitrary~$\alpha > p_c$. Given~$(\xi_1,\xi_2) \in \{0,1\}^{\Z^d} \times \{0,1\}^{\Z^d}$, define
$$\tilde \xi_M (x) := \mathds{1}\{ \xi_2 \in \theta_{Mx} E^{(1)}_M, (\xi_1,\xi_2) \in \theta_{Mx} E^{(2)}_M \}, \quad x \in \Z^d.$$
It follows from the definition of the sets~$E^{(1)}_M$ and~$E^{(2)}_M$ that if there exists an infinite connected path~$\tilde \gamma$ of open sites in~$\tilde \xi_M$, then there exists an infinite connected path~$\gamma$ of open sites in~$\xi_1$. This implies
\begin{equation}\label{ineq_tilde_mu_stir}
    \mu_{\alpha,\mathsf v}(\mathrm{Perc}) \ge \tilde \mu_{\alpha,\beta,M,\mathsf v}(\mathrm{Perc}), \text{ for any } \alpha > \beta,
\end{equation}
where~$\tilde \mu_{\alpha,\beta,M,\mathsf v}$ is the distribution of~$\tilde \xi_M$ when~$(\xi_1,\xi_2)$ is sampled from~$\mu_{\alpha,\beta,\mathsf v}$ defined in Section~\ref{section_5}.

We want to prove that there is a choice of~$L$ and~$\ell$ in~\eqref{scales_mathcal_L0},~$M \in \N$ and~$\alpha > \beta$ such that
\begin{equation}\label{eq_implies_part_2_theorem}
\begin{aligned}
    &\text{there exist } \mathsf v_0 > 0 \text{ such that, for any } \mathsf v \ge \mathsf v_0,
    \\
    &\hspace{5cm} \tilde \mu_{\alpha,\beta,M, \mathsf v} \big( B(L_N - 2) \xleftrightarrow{\; \ast (1- \tilde \xi) \;} B(2L_N)^c \big) \leq 2^{-2^N}, \quad N \in \N.
\end{aligned}    
\end{equation}
As observed in~\eqref{eq_part_2_imply_noperc}, this implies that if~$\mathsf v \ge \mathsf v_0$ we have~$\tilde \mu_{\alpha,\beta,M,\mathsf v}(\mathrm{Perc}) = 1$. Consequently, by~\eqref{ineq_tilde_mu_stir}, we also have~$\mu_{\alpha,\mathsf v}(\mathrm{Perc}) = 1$.

We now begin the proof of~\eqref{eq_implies_part_2_theorem} by setting~$L=1$. With this choice we have that:
$$\{ B(\mathcal T(m), L-2)\xleftrightarrow[ ]{\ast(1-\tilde \xi)}B(\mathcal T(m), 2L)^c\}=\{\tilde \xi(\mathcal T(m))=0\}.$$

For any~$N \in \mathbb{N}$, we apply~\eqref{path_anulus} followed by a union bound and~\eqref{bound_for_Lambda_N} to obtain 
\begin{equation}\label{eq_tildemu}
    \begin{aligned}
        & \tilde \mu_{\alpha,\beta,M,\mathsf v}  \big( B(L_N-2)  \xleftrightarrow[ ]{\ast(1-\tilde \xi)} B(2L_N) \big) \\[0.2cm]
        &  \leq (C_d \cdot \ell^{2d-2})^{2^N} \cdot \max_{\mathcal T \in \mathcal P_N} \tilde \mu_{\alpha,\beta,M,\mathsf v} \Big( \bigcap_{m \in T_{(N)}} \{ \tilde \xi( \mathcal T(m) ) = 0 \} \Big)\\[0.2cm]
        & = (C_d \cdot \ell^{2d-2})^{2^N} \cdot \max_{\mathcal T \in \mathcal P_N} \mu_{\alpha,\beta, \mathsf v} \Big(  \bigcap_{m \in T_{(N)}} \big\{ \xi_2 \notin \theta_{ M \cdot \mathcal T(m) } E^{(1)}_M \big\} \cup \big\{ (\xi_1,\xi_2) \notin \theta_{ M \cdot \mathcal T(m) }E^{(2)}_M \big\} \Big) 
    \end{aligned}
\end{equation}
where the equality follows from the definition of~$\tilde \xi_M$. By monotonicity with respect to set inclusion, the measure on the right-hand side of~\eqref{eq_tildemu} is bounded above by
\begin{equation}\label{eq_tildemu_2}
\sum_{B \subseteq T_{(N)}} \mu_{\alpha,\beta, \mathsf v} \Big( \bigcap_{m \in B} \theta_{M \cdot \mathcal T(m)} \big( \{0,1\}^{B(M)} \times \big(E^{(1)}_M\big)^c \big) \cap  \bigcap_{m \in T_{(N)}\setminus B} \theta_{M \cdot \mathcal T(m)} \big( E^{(2)}_M\big)^c \Big).
\end{equation}

Fix~$N \in \N$, and~$\mathcal T \in \mathcal P_N$, a proper embedding with spatial boxes of scale~$L=1$ and~$\ell$, then the map~$\mathcal T_M: T_N \to \Z^d$ defined by~$\mathcal T_M(m) = M \cdot \mathcal T(m)$ is a proper embedding with spatial boxes of scale~$L=M$ and~$\ell$. Note that~$\{0,1\}^{B(M)} \times (E^{(1)}_M)^c \in \mathcal G^{(2)}_{B(M)}$ is an increasing-decreasing event and~$E^{(2)}_M \in \mathcal G^{(2)}_{B(M)}$ is a decreasing-increasing event, so we can fix~$B \subseteq T_{(N)}$ and apply Lemma~\ref{lemma_no_same_E} to obtain that the measure in~\eqref{eq_tildemu_2} is bounded above by
\begin{equation}\label{eq_translation_E_M_2}
c(\ell,M/2,\widehat{\theta},N) \cdot \sum_{A \subseteq T_{(N)}} \left( \prod_{ \substack{m \in T_{(N)}   \setminus A \\ m \in B }}\Psi \big( \{0,1\}^{B(L)} \times \big(E^{(1)}_M\big)^c \big) \cdot \prod_{ \substack{m \in T_{(N)}  \setminus A \\ m \in T_{(N)} \setminus B }}\Psi \big(E^{(2)}_M \big)^c \right) \cdot (\Psi_1)^{|A|}.
\end{equation}

Now, setting
\begin{equation}\label{ell_def_M}
    \ell = M^{16d} \quad \text{ and } \quad \mathsf v_0 = \ell \exp(\ell) - 1,
\end{equation} 
we show that $M$ can be taken large enough so that~\eqref{eq_implies_part_2_theorem} holds. We begin by bounding~\eqref{eq_translation_E_M_2} from above. We choose~$M_{(1)}$ large enough such that for~$M \ge M_{(1)}$,
\begin{equation}\label{c_upper_bound_percolation}
    c(\ell,M/2,\widehat{\theta},N) \stackrel{\eqref{c_ell_L_0_alpha_def},\eqref{ell_def_M}} =  \exp\{5c_\mathrm{nn} \cdot \widehat{\theta}^{-2} \cdot |B(M)|^2 \cdot 2^N \cdot M^{4d(2-d)} \} \leq \exp\{ 2^N \}.
\end{equation}

Recall that we have fixed $\alpha>p_c$ and the definition of~$q$ from~\eqref{q_def}. We set~$\beta = (3\alpha + p_c)/4$. There exists~$M_{(2)}$ such that for~$M \ge M_{(2)}$,
$$ q =2^{d/2} \cdot 3^d \cdot c_d \cdot |B(M)| \cdot M^{-4d^2} < \frac{\alpha - p_c}{16}, $$
which implies that
$$\alpha(1-q) > \frac{7\alpha + p_c}{8} > \beta +q -\beta q >\beta(1-q) > \frac{\alpha + p_c}{2}>p_c,$$
where the first, second, and fourth inequality follow from the constraints of~$q$ and the value of~$\beta$. From the monotonicity of the events and Remark~\ref{rmk_inclusion} (the condition~$\alpha(1-q) \ge \beta + q -\beta q$ is required for the remark to apply) we have, 
\begin{equation}\label{max_pi_1}
\begin{aligned}
\pi_{\alpha+q-\alpha q, \; \beta(1-q)} &\big( \{0,1\}^{\Z^d} \times (E^{(1)}_M )^c \big) \leq  \pi_{\frac{\alpha+p_c}{2}}\big( \big( E^{(1)}_M \big)^c \big) \leq \pi_{\frac{\alpha+p_c}{2}}(E_M^c ), \\[2mm]
\pi_{\alpha(1-q), \; \beta+q-\beta q} &\big((E^{(2)}_M)^c  \big) \leq \pi_{\alpha(1-q), \; \frac{7\alpha + p_c}{8}} \big( \big(E^{(2)}_M \big)^c \big) \\[2mm]
&\leq \pi_{\alpha(1-q), \; \frac{7\alpha + p_c}{8}} \big(\{0,1\}^{\Z^d} \times E^c_M \big)  + \pi_{\alpha(1-q), \; \frac{7\alpha + p_c}{8}} \big((E^{(2)}_M)^c\cap \big( \{0,1\}^{\Z^d} \times E_M \big) \big) \\[2mm]
& = \pi_{\frac{7\alpha + p_c}{8}}( E_M^c),
\end{aligned}
\end{equation}
where in the last equality we used the fact that the measure $\pi_{\alpha(1-q), \; \frac{7\alpha + p_c}{8}}$ puts zero mass on configurations $(\xi', \xi)$ with $\xi'\ngeq \xi$. From Theorem~\ref{thm_supercritical}, there exists~$\mathsf a = \mathsf a(\alpha,d)$ and~$\mathsf b = \mathsf b(\alpha,d)$ such that for each~$M \ge M_{(2)}$,
\begin{equation}\label{max_pi_2}
\max \left\{ \pi_{\frac{\alpha+p_c}{2}}(E_M^c ), \;  \pi_{\frac{7\alpha+p_c}{8}}(E_M^c ) \right\} \stackrel{\eqref{eq_pisztora}} \leq \mathsf a \cdot \exp \{ - \mathsf b \cdot M \}        
\end{equation}
Hence, there exists~$M_{(3)} \ge M_{(2)}$ such that for~$M \ge M_{(3)}$,
\begin{align*}
    & \max \big\{ \Psi\big(  \{0,1\}^{\Z^d} \times \big( E^{(1)}_M\big)^c \big),\Psi\big( \big(E^{(2)}_M \big)^c \big) \big\}  \\[2mm] 
    &\hspace{10mm} \stackrel{\eqref{def_Psi_incrdecr},\eqref{ell_def_M},\eqref{max_pi_1},\eqref{max_pi_2}} \leq \mathsf a \cdot \exp \{ - \mathsf b \cdot M \} + a \cdot |B(M)|^2 \exp \Big\{ -\frac{b \cdot M^{4d}}{|B(M)|} \Big\} \leq 2\mathsf a \cdot \exp \{ - \mathsf b \cdot M \} ,\\[0.2cm]
    & \Psi_1 \stackrel{\eqref{Psi_1_def},\eqref{ell_def_M}} = |B(M)| \cdot \Big( 2d \Big( 1 + \frac{d}{4} \Big)^{-M^{16d}/8} + \exp\{-M^{16d}\} \Big) \leq \mathsf a \cdot \exp \{ - \mathsf b \cdot M \}.
\end{align*}

We now upper bound~\eqref{eq_translation_E_M_2} using~\eqref{c_upper_bound_percolation} and the two inequalities above. This yields that, for~$M \ge \max\{M_{(1)},M_{(3)}\}$, the term in~\eqref{eq_translation_E_M_2} is smaller than 
$$ \exp\{2^N\} \cdot \sum_{A \subseteq T_{(N)}} (2\mathsf a \exp\{- \mathsf b M\})^{|T_{(N)} \setminus A|} \cdot (\mathsf a \exp\{- \mathsf b M\})^{|A|} \leq (4\mathsf a \exp\{1- \mathsf b M\} )^{2^N}. $$
Then, since this bound holds for each~$B \subseteq T_{(N)}$, the term in~$\eqref{eq_tildemu_2}$ is bounded above by~$(8\mathsf a \exp\{1- \mathsf b M\} )^{2^N}$. Plugging this into~\eqref{eq_tildemu}, we conclude that
\begin{align*}
    \tilde \mu_{\alpha,M,\mathsf v}  \big( B(L_N-2) & \xleftrightarrow[]{\; \ast (1-\tilde \xi \;)} B(2L_N)^c \big) \\[.2cm]
    & \leq (C_d \cdot M^{32d^2-32d})^{2^N} \cdot(8\mathsf a \exp\{1- \mathsf b M\} )^{2^N} \\[0.2cm]
    & = \big( 8 \mathsf a C_d \cdot M^{32d^2-32d} \cdot  \exp\big\{ 1 - \mathsf b \cdot M \big\} \big)^{2^N},
\end{align*}
which proves~\eqref{eq_implies_part_2_theorem} by taking~$M$ large enough.
\end{proof}

\section{Appendix - Proof of Proposition~\ref{prop_couple_irw_ep}}\label{Section_Appendix}
We define an \textit{instruction manual} for a random walk on $\Z^d$ as a pair~$(\mathcal{T},\mathsf m)$, where $\mathcal{T}$ is a rate-1 Poisson point process on~$[0,\infty)$ and $\mathsf m: \mathcal T \to \{\pm e_i: i=1,\dots,d\}$ is a random function where~$e_i$ is the $i$th standard basis vector of~$\R^d$ and, conditionally on~$\mathcal T$, the values~$\{\mathsf m (t): t \in \mathcal T\}$ are independent and uniformly distributed on~$\{\pm e_1,\dots,\pm e_d\}$. To define a random walk with jump rate~$1$ started from $x \in \Z^d$ with this instruction manual we enumerate~$\mathcal T = (t_1,t_2,\dots)$ in increasing order and set~$X_t := x$ for~$ t \in [0,t_1)$. Assuming that $X_{t^-_n}$ has been defined for some $n \in \N$ we set~$X_t := X_{t^-_n} + \mathsf m (t_n)$ for~$ t \in [t_n, t_{n+1})$. 

Given $k \in \N$ for a system of $k$ independent random walks~$(X^1_t,\dots, X^k_t)_{t \ge 0}$ of jump rate~1 started from~$(x_1,\dots,x_k) \in (\Z^d)^k$ on the probability space~$\widehat{\Prob}$ we consider the independent instruction manuals~$(\mathcal T_1, \mathsf m_1), \dots, (\mathcal T_k, \mathsf m_k)$. 

We now construct a system of~$k$ stirring particles~$(W^1_t,\dots,W^k_t)_{t \ge 0}$ started from~$(x_1,\dots,x_k)$ using the same instruction manuals. Given~$(z_1,\dots,z_k) \in (\Z^d)^k$, for~$i,j \in \{ 1,\dots,k \}$ with~$i<j$ we write
$$ (z_1,\dots,z_k)^\mathrm{stir}_{i,j} := (z_1,\dots,z_{i-1},z_j,z_{i+1},\dots,z_{j-1},z_{i},z_{j+1},\dots,z_k).$$

Enumerate~$\cup^k_{i=1} \mathcal T_i = (t_1,t_2,\dots)$. We set~$(W^1_t,\dots,W^k_t) := (x_1,\dots,x_k)$ for~$t \in [0,t_1)$. Assuming that~$(w_1,\dots,w_k) = (W^1_{t^-_n}, \dots, W^k_{t^-_n})$ has been defined for some~$n \in \N$ , we set for~$t \in [t_n,t_{n+1})$
\begin{equation}
    (W^1_t, \dots, W^k_t) := \left\{ \begin{array}{ll}
        (w_1,\dots,w_i + \mathsf{m}_i(t_n),\dots, w_k)   & \text{if } t_n \in \mathcal T_i \text{ and }  \\[0.1cm]
        & w_i + \mathsf{m}_i(t_n) \notin \{w_1,\dots, w_k\},  \\[0.2cm]
        (w_1,\dots,w_k)^\mathrm{stir}_{i,j} & \text{if } t_n \in \mathcal{T}_i \text{ and } w_i + \mathsf{m}_i(t_n) = w_j, \; i<j,\\[0.2cm]
        (w_1,\dots,w_k)   & \text{otherwise.}
    \end{array} \right.
\end{equation}
To describe the discrepancies of the particles $(W^i_t)_{t \ge 0}$ and $(X^i_t)_{t \ge 0}$ for each $i$ we introduce the process~$\left(D_t(i,j)\right)_{t \ge 0}$ for each $j < i$ defined as follows. We set $D_t(i,j) := 0$ for~$t \in [0,t_1)$. Assume that~$D_{t^-_n}$ has been defined for some $n \in \N$. 
\begin{itemize}
    \item If~$|W^i_{t^-_n}-W^j_{t^-_n}|_1 \neq 1$, then set~$D_t(i,j) := D_{t_n^{-}}(i,j)$ for~$t \in [t_n,t_{n+1})$.
    \item If~$|W^i_{t^-_n}-W^j_{t^-_n}|_1=1$, then set for~$t \in [t_n,t_{n+1})$
$$ D_t(i,j) := D_{t^-_n}(i,j) + \big[ W^i_{t_n} - W^i_{t^-_n} \big] - \big[ X^i_{t_n} - X^i_{t^-_n} \big].$$
\end{itemize}
Hence,
\begin{equation}\label{eq_D_t_D_t_n}
    D_{t_n}(i,j) - D_{t^-_n}(i,j) = \left\{ \begin{array}{ll}
        -\mathsf{m}_i(t_n) & \text{ if } t_n \in \mathcal T_i \text{ and } W^i_{t^-_n} + \mathsf{m}_i(t_n) = W^j_{t^-_n},  \\[0.2cm]
        - \mathsf{m}_j(t_n) & \text{ if } t_n \in \mathcal T_j \text{ and } W^j_{t^-_n} + \mathsf{m}_j(t_n) = W^i_{t^-_n},\\[0.2cm]
        0 & \text{ otherwise.}
    \end{array} \right.
\end{equation}    
From the construction, we can check that
$$W^i_t - X^i_t = \sum_{j<i} D_t(i,j) \text{ for all } i=1,\dots, k \text{ and } t \ge 0.$$
Then, for any~$t \ge 0$
\begin{equation}\label{eq_i_X_T_W_T}
    \widehat{\Prob}( \exists \; i \in \{1,\dots,k\} : |W^i_t - X^i_t| > r ) \leq  \sum_{1 \leq j < i \leq k} \widehat{\Prob}( |D_t(i,j)| > r/k).
\end{equation}
To obtain an upper bound on the probability on the right-hand side of this inequality we use the following martingale result that appears as Theorem~2.4 of~\cite{RathValesin2017}, which is an adaptation of Theorem~26.17 of~\cite{Kallenberg2002}.
\begin{theorem}\label{thm_kallemberg}
Let $S \in [0,\infty]$. Let~$(N_t)$ be a square-integrable càdlàg martingale with \textit{predictable quadratic variation} $\langle N \rangle_S \leq \sigma^2$ almost surely for some for some~$\sigma^2 \in (0,\infty)$.  Assume that the jumps of $N$ are almost surely bounded by $\Delta \in (0,\sigma]$. Then we have
$$ \Prob \big[ \max_{0\leq t \leq S}N_t - N_0 > r \big] \leq \exp \Big\{ -\frac{1}{2} \frac{r}{\Delta}  \ln \Big(1+ \frac{r \Delta}{\sigma^2} \Big) \Big\}, \quad r \ge 0. $$    
\end{theorem}
We use this theorem to prove the following lemma.
\begin{lemma}\label{lemma_prob_d_t_exp}
    There exist dimension-dependent constants~$a,b > 0$ such that for any~$u \ge 0$,~$j<i$ and any initial positions~$W^i_0,W^j_0$ we have that
    $$\sup_{t \ge 0} \; \widehat{\Prob}( |D_t(i,j)| > u ) \leq a \cdot \exp\{-b \cdot u\}.$$
\end{lemma}
This lemma together with~\eqref{eq_i_X_T_W_T} proves Proposition~\ref{prop_couple_irw_ep}.
\begin{proof}[Proof of Lemma~\ref{lemma_prob_d_t_exp}]
We fix $i$ and $j<i$. Let~$D_t(i,j) = (D^1_t,\dots,D^d_t)$ and
\begin{equation}\label{I_k_t_rv}
        I^k_t := \mathrm{Leb} \big( \big\{ s \in [0,t]: W^i_s - W^j_s \in \{e_k,-e_k\} \big\} \big), \quad k =1,\dots,d.
\end{equation}
For each~$k$, the coordinate process~$(D^k_t)_{t \ge 0}$ evolves in~$\Z$ in the following way. Let us suppose~$W^i_{t_n^-} - W^j_{t_n^-} \in \{e_k,-e_k\}$, say~$W^j_{t_n^-} = W^i_{t_n^-} + e_k$. According to~\eqref{eq_D_t_D_t_n}, we have that
$$ D^k_{t_n} = \left\{ \begin{array}{ll}
    D^k_{t_n^-} - 1, & \text{if } t_n \in \mathcal T_i \text{ and } \mathsf m_i(t_n) = e_k, \\[2mm]
    D^k_{t_n^-} + 1, & \text{if } t_n \in \mathcal T_j \text{ and } \mathsf m_j(t_n) = -e_k, \\[2mm]
    D^k_{t_n^-}, & \text{otherwise.}
\end{array} \right. $$
Hence, given~$I^k_t$ this process has the same distribution as~$Y_{I^k_t}$, where~$(Y_t)_{t \ge 0}$ is a continuous-time simple random walk on~$\Z$ with jump rate~$1/d$ started from the origin. Recall that the predictable quadratic variation of a simple random walk is the deterministic process~$t$. By a union bound followed by an application of Theorem~\ref{thm_kallemberg} with~$S=I^k_t$,~$\sigma^2 = I^k_t/d$ and~$\Delta=1$ we obtain that 
\begin{equation}\label{D_T_ineq}
    \begin{aligned}
        \widehat{\Prob}(|D_t(i,j)| >  u) & \leq \sum^d_{k=1} 2 \cdot \widehat{\E} \big[ \widehat{\Prob} (D^k_t > u\mid I^k_t) \big] \\
        & \leq \sum^d_{k=1} 2 \cdot \widehat{\E} \Big[\exp \Big\{ -\frac{1}{2}u \ln \Big( 1+\frac{d\cdot u}{I^k_t } \Big) \Big\} \Big]\\
        & \leq 2d \cdot \exp \Big\{ -\frac{1}{2}u\ln( 2 ) \Big\} + 2d \cdot \widehat{\Prob} \big( I_t \ge d \cdot u \big), 
    \end{aligned}
\end{equation}
where
$$I_t := \sum^d_{k=1}I^k_t = \mathrm{Leb} \big( \big\{ s \in [0,t]: |W^i_s - W^j_s|_1 = 1 \big\} \big).$$ 
We now focus on~$I_t$. Note that the return time of $|W_s^i-W_s^j|_1$ to $1$ is the same as if the particles were independent random walks with rate $1$ and that the holding time when $|W_s^i-W_s^j|_1 = 1$ is exponentially distributed with parameter~$\tfrac{2d-1}{2d}$. Based on this observation, we now introduce the following random variables. Let $(\sigma_i)_{i \in \N}$ and $N$ be independent random variables such that 
$$ \sigma_i \stackrel{\mathrm{d}}{=}  \text{Exp}\big( \tfrac{2d-1}{2d} \big) \text{ for all } i\in\N \quad \text{ and } \quad N \stackrel{\mathrm{d}}{=}  \text{Geo}(p),  $$
where~$p$ denotes the probability that two independent random walks, having initially been~$\ell^1$-neighbors, become~$\ell^1$-neighbors again after moving apart, that is,
$$ p = \Prob( \exists t \ge 0: |X_t - Y_t|_1 = 1, |X_{t^-} - Y_{t^-}|_1 = 2 \;  \mid |X_0 - Y_0|_1 = 1).$$
By transience we have $p = p(d) \in (0,1)$. Moreover,~$I_s \leq \sum^N_{i = 1} \sigma_i$ for all $s \ge 0$ and for any initial locations~$W^i_0, W^j_0$.

Letting~$c = \tfrac{1}{2} \cdot (2d-1)\cdot (1-\sqrt{p})$, we have
\begin{equation}\label{I(t)_ineq_v}
    \begin{aligned}
         \widehat{\Prob} \big(I(t) \ge d\cdot u \big) & \leq \exp \big( - c\cdot u \big) \cdot \widehat{\E}\Big[\exp\Big( \frac{c}{d} \cdot \sum^N_{i = 1} \sigma_i \Big) \Big] \\[0.2cm]
         & = \exp \big(-c \cdot u \big) \cdot \sum_{n \ge 1} \widehat{\Prob}(N=n) \cdot \Big( \widehat{\E} \exp \Big( \frac{c}{d} \cdot \sigma_1 \Big) \Big)^n \\[0.2cm]
         & = \exp \big(-c \cdot u \big) \cdot \sum_{n \ge 1} (1-p) \cdot p^n \cdot \Big( \frac{1}{\sqrt{p}} \Big)^n \leq \exp \big(-c \cdot u \big) \cdot \sum_{n \ge 1} \big( \sqrt{p} \big)^n.
    \end{aligned}
\end{equation}
Therefore, plugging~\eqref{I(t)_ineq_v} back in~\eqref{D_T_ineq} we have the desired result with~$b = \min \big\{ \tfrac{\ln(2)}{2}, c \big\}$ and~$a=2d \big(1+ \sum_{n \ge 1} (\sqrt{p})^n \big)$.
\end{proof}

\subsection*{Acknowledgments}
J.A. was supported by FAPESP (grants 2023/13453-5 and 2025/02707-1) and by the CogniGron research center and the Ubbo Emmius Funds (Univ. of Groningen); the author is thankful for this support.

\bibliographystyle{alpha}
\bibliography{Ref.bib}

@book{Lawler2010,
  title={Random Walk: A Modern Introduction},
  author={Lawler, G. and Limic, V},
  year={2010},
  publisher={Cambridge University Press}
}

@book{Swart2017,
    author = {J. Swart},
    title = {A Course in Interacting Particle Systems},
    publisher = {In: arXiv:1703.10007},
    year ={2017} 
}

@article{CliffordSudbury73,
    author = {Clifford, P. and Sudbury A.},
    title = {A model for spatial conflict},
    journal = {Biometrika},
    volume = {60},
    pages = {581-588},
    year = {1973}
}

@book{Grimmett99,
  title={Percolation. Grundlehren der Mathematischen Wissenschaften [Fundamental Principles of Mathematical Sciences]},
  author={Grimmett, G.},
  year={1999},
  publisher={Springer-Verlag, New York}
}

@book{Liggett2005,
  title={Interacting Particle Systems. Grundlehren der Mathematischen Wissenschaften [Fundamental Principles of Mathematical Sciences]},
  author={Liggett, T. M.},
  year={1985},
  publisher={Springer, New York}
}

@article{RathValesin2017,
    title = {Percolation on the Stationary Distributions of the Voter Model},
    author = {Ráth, B. and Valesin, D.},
    journal = {Ann. Probab.},
    volume = {45},
    number = {no. 3},
    pages = {1899-1951},
    year = {2017}
}

@book{Kallenberg2002,
  title={Foundations of Modern Probability.},
  author={Kallenberg, O.},
  year={2002},
  publisher={Springer New York, NY. 2nd Edition}
}

@article{Astoquillca24,
    author = {Astoquillca, J.},
    title = { {On the Stationary Measures of Two Variants of the Voter Model} },
    journal = {Journal of Theoretical Probability},
    year = {2026},
    volume = {39},
    number = {34},
}

@article{HolleyLiggett75,
    author ={Holley, R. and Liggett, T. M.},
    title = {Ergodic theorems for weakly interacting systems and the voter model},
    journal = {Annals of Probability},
    volume = {3},
    pages = {643-663},
    year = {1975}
}

@article{Rath15,
    author ={Ráth, B.},
    title = {A short proof of the phase transition for the vacant set of random interlacements},
    journal = {Electronic Communications in Probability},
    volume = {20},
    year = {2015}
}

@article{RathValesin17OnT,
    author ={Ráth, B. and Valesin, D.},
    title = {On the threshold of spread-out voter model percolation},
    journal = {Electronic Communications in Probability},
    volume = {22},
    pages = {1-12},
    year = {2017}
}

@article{DuminilTassion16,
    author ={Duminil-Copin, H. and Tassion, V.},
    title = {A New Proof of the Sharpness of the Phase Transition for Bernoulli Percolation and the Ising Model},
    journal = {Commun. Math. Phys.},
    volume = {343},
    pages = {725-745},
    year = {2016}
}

@article{Cerf2000,
    author ={Cerf, R.},
    title = {Large deviations for three dimensional supercritical percolation},
    journal = {vol. 267 of Astérisque, SMF},
    volume = {},
    pages = {},
    year = {2000}
}

@article{Pisztora96,
    author ={Pisztora, A.},
    title = {Surface order large deviations for Ising, Potts and percolation models},
    journal = {Probab. Th. Rel. Fields},
    volume = {104},
    pages = {427–466},
    year = {1996}
}

@article{DeMasiFerrariLebowitz96,
    author ={De Masi, A. and Ferrari, P. and Lebowitz, J.},
    title = {Reaction-diffusion equations for interacting particle systems},
    journal = {J. Statist. Phys.},
    volume = {44},
    pages = {589-644},
    year = {1986}
}

@article{DurrettNeuhauser94,
    author ={Durrett, R. and Neuhauser, C.},
    title = {Particle systems and reaction-diffusion equations},
    journal = {Ann. Probab.},
    volume = {22},
    pages = {289-333},
    year = {1994}
}

@article{MiltonLandim2023,
author = {Jara, M. and Landim, C.},
year = {2023},
title = {The stochastic heat equation as the limit of a stirring dynamics perturbed by a voter model},
volume = {33},
journal = {The Annals of Applied Probability},
}

@article{Neuhauser1990,
author = {Neuhauser, C.},
title = {\text{One Dimensional Stochastic Ising Models with Small Migration}},
volume = {18},
journal = {The Annals of Probability},
number = {4},
pages = {1539-1546},
year = {1990},
}

@article{Katori94,
author = {Katori, M.},
title = {Rigorous results for the diffusive contact processes in $d \ge 3$},
volume = {27},
journal = {J. Phys. A.},
number = {22},
pages = {7327–7341},
year = {1994},
}

@article{Konno95,
author = {Konno, N.},
title = {Asymptotic behavior of basic contact process with rapid stirring},
volume = {8},
journal = {J. Theoret. Probab.},
number = {4},
pages = {833–876},
year = {1995},
}

@article{LevitValesin17,
author = {Levit, A. and Valesin, D.},
title = {Improved asymptotic estimates for the contact process
with stirring},
volume = {31},
journal = {ALEA Lat. Am. J. Probab. Math. Stat.,},
number = {2},
pages = {254–274},
year = {2017},
}

@article{MytnikShlomov21,
author = {Mytnik, L. and Shlomov, S.},
title = {\text{General Contact Process with Rapid Stirring}},
volume = {18},
journal = {ALEA Lat. Am. J. Probab. Math. Stat.,},
number = {1},
pages = {17–33},
year = {2021},
}

@article{Duminiletal23,
author = {Duminil-Copin, H. and Goswami, S. and Rodriguez, P-F. and Severo, F.},
title = {Equality of critical parameters for percolation of \text{Gaussian free field level sets}},
volume = {172},
journal = {Duke Mathematical Journal},
number = {5},
pages = {839-913},
year = {2023},
}

\end{document}